\documentclass[a4paper]{amsart}

\usepackage{amsfonts,amssymb,amscd,amsmath,latexsym,amsbsy,enumitem,stmaryrd,a4wide,verbatim,color,subfiles,relsize}
\usepackage{mwe}
\usepackage{hyperref}
\usepackage{tikz}
\usetikzlibrary{arrows,positioning,shapes.misc}
\theoremstyle{plain}
\newtheorem{theorem}{Theorem}[section]
\newtheorem{corollary}[theorem]{Corollary}

\newtheorem{lemma}[theorem]{Lemma}
\newtheorem{proposition}[theorem]{Proposition}

\theoremstyle{remark}
\newtheorem{remark}[theorem]{Remark}
\numberwithin{equation}{section}

\newcommand{\R}{\mathbb R}

\newcommand{\C}{\mathbb C}
\newcommand{\Z}{\mathbb Z}

\newcommand{\De}{\Delta}

\newcommand{\half}{\frac{1}{2}}

\newcommand{\tensor}{\otimes}
\newcommand{\rFs}[5]{\,_{#1}F_{#2} \left( \genfrac{.}{.}{0pt}{}{#3}{#4}	\ ;#5 \right)}
\newcommand{\rphis}[5]{\,_{#1}\varphi_{#2} \left( \genfrac{.}{.}{0pt}{}{#3}{#4}
\ ;#5 \right)}
\newcommand{\rphisempty}[2]{\,_{#1}\varphi_{#2}}

\newcommand{\su}{\mathfrak{su}}
\newcommand{\U}{\mathcal U}

\newcommand{\qbinom}[3]{\genfrac{[}{]}{0pt}{0}{#2}{#3}_{#1}}

\begin{document}
\title[Quantum algebra approach to rational functions of $q$-Racah type]{Quantum algebra approach to univariate and multivariate rational functions of $q$-Racah type}
\author{Wolter Groenevelt and Carel Wagenaar}
\maketitle
\begin{abstract}
	In this paper, we study rational functions of $q$-Racah type and a multivariate extension, using representation theory of $\U_q(\mathfrak{sl}_2)$. Eigenfunctions of twisted primitive elements in $\U_q(\mathfrak{su}_2)$ can be expressed in terms of $q^{-1}$-Krawtchouk polynomials. Using this, we show that overlap coefficients of solutions of a generalized eigenvalue problem (GEVP) and an eigenvalue problem (EVP) can be expressed in terms of a rational function of $\rphisempty{4}{3}$-type. With help of the quantum algebra, we derive (bi)orthogonality relations as well as a GEVP for these functions. Furthermore, using this new algebraic interpretation, we can exploit the co-algebra structure of $\U_q(\mathfrak{sl}_2)$ to find a multivariate extension of these rational functions and derive biorthogonality relations and GEVPs for the multivariate functions. Then we repeat this procedure for the non-compact quantum algebra $\U_q(\mathfrak{su}_{1,1})$, where the $q^{-1}$-Al-Salam--Chihara polynomials play the role of the $q^{-1}$-Krawtchouk polynomials. As an application of the multivariate rational functions, we show that they appear as duality functions for certain interacting particle systems.
\end{abstract}

\section{Introduction}
Finding an algebraic underpinning of special functions has been of great interest in order to obtain a better understanding of these functions (see e.g. \cite{ViKl}). For example, one can use (quantum) algebras to study orthogonal polynomials and prove fundamental properties such as the three-term recurrence relations, orthogonality and symmetries. One of these interpretations is the realization of $q$-Racah polynomials, i.e. Askey-Wilson polynomials which are orthogonal on a finite set, as overlap coefficients of eigenfunctions of some specific elements of the algebra \cite{Zhe}. That is, one has to solve two eigenvalue problems (EVPs),
\begin{align*}
	Xf &= \lambda f, \\
	Yg &= \mu g,
\end{align*}
where $X$ and $Y$ are difference operators coming from representations of specific elements of the algebra. Let us denote the family of solutions of these EVPs by $\{f_x\}$ and $\{g_y\}$ with corresponding eigenvalues $\{\lambda_x\}$ and $\{\mu_y\}$ respectively. Then the overlap coefficients are the inner product of these functions,
\[
	P(x,y)=\langle f_x, g_y\rangle.
\]
Many orthogonal polynomials of the Askey-scheme can be found in this way. \\

In recent years, there has been an interest in doing something similar for biorthogonal rational functions which can be described by a (basic) hypergeometric series. These functions are often connected to generalized eigenvalue problems (GEVPs), which are problems of the form
\[
	Xf = \lambda Z f,
\] 
where $X$ and $Z$ are linear operators. For example, when $X,Z$ are tridiagonal operators, the solutions are biorthogonal rational functions \cite{ZheBi}. Also, it seems that many rational functions can be interpreted as overlap coefficients of solutions of GEVPs. This resulted in meta-algebras encoding both ($q$-)Hahn polynomials and rational functions of ($q$)-Hahn type \cite{TVZ,VinZhe}. Very recently, a rational Wilson algebra has been introduced \cite{CTVZ}, capturing recurrence relations of Wilson's rational $\rphisempty{10}{9}$-functions. However, so far no algebra has been found from which the fundamental properties of these rational functions can be elegantly derived. This paper aims to bridge this gap for a special case of Wilson's rational functions: the rational functions of $q$-Racah type.  Using $\U_q(\mathfrak{sl}_2)$, we find a quantum algebraic interpretation of these rational $\rphisempty{4}{3}$-functions. We show that these functions are overlap coefficients of solutions to an EVP and a GEVP. That is, it is the inner product of functions $f$ and $g$ which solve
\begin{align*}
	Xf &= \lambda Zf, \\
	Yg &= \mu g,
\end{align*}
for certain linear operators $X$, $Y$ and $Z$. Using $\U_q(\mathfrak{sl}_2)$ and its $*$-structures, we derive the (bi)orthogonality and recurrence relations of these functions. We do this for the compact case $\U_q(\mathfrak{su}_2)$, which leads to a finite orthogonality, and the non-compact case $\U_q(\mathfrak{su}_{1,1})$, which gives a countably infinite orthogonality. \\

Using this new algebraic interpretation of rational functions, we can also exploit the co-algebra structure of $U_q(\mathfrak{sl}_2)$ to find a multivariate extension of these functions and prove its (bi)orthogonality and recurrence relations. These multivariate functions have a nested structure similar to the Tratnik-type orthogonal polynomials \cite{GRmulti}. Finding algebraic interpretations of multivariate extensions of orthogonal polynomials has recently been of great interest (see e.g. \cite{CFR,GIV,Gr21}), but so far this has not been done in a similar way for biorthogonal rational functions.\\

Let us describe the structure of this paper in a bit more detail. In section \ref{sec:sum} we prove a formula connecting the summation of two $\rphisempty{3}{2}$'s with a $\rphisempty{4}{3}$. In essence, this shows that the inner product of two (different) $q^{-1}$-Krawtchouk (or $q^{-1}$-Al-Salam--Chihara) polynomials are rational functions of $\rphisempty{4}{3}$-type. Then, in section \ref{sec:su2rat}, we formulate a GEVP and EVP involving twisted primitive elements of the quantum algebra $\U_q(\mathfrak{su}_2)$. We rewrite the GEVP to an EVP and use that $q$-Krawtchouk polynomials are eigenfunctions of twisted primitive elements to solve these (G)EVPs. Using the summation formula of the previous section, we show that the overlap coefficients of the solutions to these (G)EVPs are rational $\rphisempty{4}{3}$-functions of $q$-Racah type. Then, using the GEVP and orthogonality of the $q$-Krawtchouk polynomials, we derive a GEVP and (bi)orthogonality of the rational functions of $q$-Racah type. Afterwards, we use the co-product structure of $\U_q(\mathfrak{su}_2)$ to find a multivariate extension of these rational functions and derive (bi)orthogonality as well as recurrence relations these functions satisfy. 
In section \ref{sec:su11rat} we repeat the steps of the previous section, but then for the non-compact case $\U_q(\mathfrak{su}_{1,1})$. Here, the $q^{-1}$-Al-Salam--Chihara polynomials play the role of the $q$-Krawtchouk polynomials. We end this article by showing an application of these rational functions. We show that they appear as self-duality functions for the interacting particle systems dynamic ASEP and dynamic ASIP.

\subsection{Preliminaries \& notations}
	Let $q>0$ such that $q\neq 1$. For aesthetic reasons, we use the following notation,
	\[
		[t]_q=\frac{q^t-q^{-t}}{q-q^{-1}},
	\]
	\[
		\{t\}_q = \frac{q^t+q^{-t}}{q+q^{-1}},
	\]
	the first satisfying 
	\[
		\lim\limits_{q\to 1}[t]_q = t.
	\]
	For linear operators $A,B$, we define the $q$-commutator by
	\[
		[A,B]_q = qAB - q^{-1}BA.
	\]
	Besides this, we use standard notation for $q$-shifted factorials and $q$-hypergeometric functions as in \cite{GR}. In particular, $q$-shifted factorials are given by
	\[
	(a;q)_n = (1-a)(1-aq) \cdots (1-aq^{n-1}), 
	\]
	with $n \in \Z_{\geq0} $ and the convention $(a;q)_0=1$. For $q \in (0,1)$, we define
	\[
	(a;q)_\infty = \lim_{n \to \infty} (a;q)_n.
	\] 
	For $n \in \Z_{\geq 0}$ and $j=0,\ldots,n$, the $q$-binomial coefficient is given by
	\[
	\qbinom{q}{n}{j} = \frac{	(q;q)_n}{	(q;q)_j	(q;q)_{n-j}}.
	\]
	The $q$-hypergeometric series $_{r+1}\varphi_r$ is given by
	\[
	\rphis{r+1}{r}{a_1,\ldots,a_{r+1}}{b_1,\ldots,b_r}{q,z} = \sum_{n=0}^\infty \frac{(a_1;q)_n \cdots (a_{r+1};q)_n}{(b_1;q)_n \cdots (b_r;q)_n} \frac{z^n}{(q;q)_n}.
	\]
	If for some $j$ we have $a_j=q^{-N}$ with $N \in \Z_{\geq 0}$, the series terminates after $N+1$ terms, since $(q^{-N};q)_n=0$ for $n>N$.\\
	
	Moreover, we make intensive use of $\U_q\big(\mathfrak{sl}_2\big)$, the quantised universal enveloping algebra of the Lie algebra $\mathfrak{sl}_2$. This is the unital, associative, complex algebra generated by $K$, $K^{-1}$, $E$, and $F$, subject to the relations
	\begin{align}
		K K^{-1} = 1 = K^{-1}K, \quad KE = qEK, \quad KF= q^{-1}FK, \quad EF-FE =\frac{K^2-K^{-2}}{q-q^{-1}}.\label{eq:UqRelations}
	\end{align}
	The comultiplication $\De:\U_q\to\U_q\otimes\U_q$ is an algebra homomorphism defined on the generators by
	\begin{equation} \label{eq:comult}
		\begin{split}
			\De(K) = K \tensor K,&\quad  \De(E)= K \tensor E + E \tensor K^{-1}, \\
			\De(K^{-1}) = K^{-1} \tensor K^{-1},&\quad  \De(F) = K \tensor F + F \tensor K^{-1}.
		\end{split}
	\end{equation}
	Throughout this paper we will use two different $*$-structures on $\U_q\big(\mathfrak{sl}_2\big)$, giving either the compact form $\U_q\big(\mathfrak{su}_2\big)$ or the non-compact form $\U_q\big(\mathfrak{su}_{1,1}\big)$. These $*$-structures will be introduced in their sections. In both cases, $\Delta$ is a $*$-algebra homomorphism.
	
\section{A summation formula for $q$-Krawtchouk and Al-Salam--Chihara polynomials} \label{sec:sum}
	Crucial for our results is the following summation formula, which connects two polynomials which can be written as $\rphisempty{3}{2}$'s in base $q^{-1}$ to a $\rphisempty{4}{3}$-rational function. We will apply this formula to $q^{-1}$-Krawtchouk polynomials and $q^{-1}$-Al-Salam--Chihara polynomials in sections \ref{sec:su2rat} and \ref{sec:su11rat} respectively.
	\begin{lemma}\label{lem:summation3phi2to4phi3}
		Let $x,y\in \Z_{\geq 0}$ and either $|bcd|<1$ and $0<q<1$ or $a=q^{-N}$ for some $N\in\Z_{\geq 0}$. Then we have the following summation formula between $\rphisempty{3}{2}$'s in base $q^{-1}$ and a rational function which is a $\rphisempty{4}{3}$ of $q$-Racah type,
		\begin{align*}
			\begin{split}
				\frac{(abcd ;q)_\infty}{(bcd;q)_\infty} &\frac{(bdq^{-y}/c;q)_{y}}{(abcd;q)_{y}}\frac{(q^{-x}/b^2;q)_{x}}{(a;q)_{x}}\rphis{4}{3}{q^{-x},ab^2 q^{x},bcd,bd/c}{b^2q,bdq^{-y}/c,abcdq^{y}}{q;q} \\
				= &\sum_{n=0}^\infty (bcd)^n\frac{(a;q)_n}{(q;q)_n}\rphis{3}{2}{q^{n}, q^{x}, q^{-x}/ab^2}{1/a,0}{q^{-1},q^{-1}} \rphis{3}{2}{q^{n}, q^{y}, q^{-y}/ac^2}{1/a,0}{q^{-1},q^{-1}}.\end{split}
		\end{align*}
	\end{lemma}
	\begin{proof}
		The proof consists of writing out the $\rphisempty{3}{2}$'s as sums, say over $j_1=0,...,x$ and $j_2=0,...,y$. Then we interchange the summations such that we first sum over $n$, then over $j_2$ and lastly over $j_1$. We will see that we can explicitly compute the sums over $n$ and $j_2$ and end up with the desired $\rphisempty{4}{3}$. Let us first rewrite the $\rphisempty{3}{2}$ such that the dependence on $n$ is simpler. The $\rphisempty{3}{2}$ in base \\$q^{-1}$ can be written into a $\rphisempty{3}{1}$ in base $q$, using
		\begin{align}
			(a;q^{-1})_j = (-a)^j q^{-\half j(j-1)}(1/a;q)_j.
		\end{align} 
		Using this, we obtain
		\begin{align*}
			\rphis{3}{2}{q^{n}, q^{x}, q^{-x}/ab^2}{1/a,0}{q^{-1},q^{-1}} &= \rphis{3}{1}{q^{-n},q^{-x},ab^2q^{x}}{a}{q,q^{n}/b^2}  \\
			&=(-1/b^2)^xq^{-\half x(x+1)}\frac{(b^2q;q)_x}{(a;q)_x} \rphis{2}{1}{q^{-x},ab^2q^{x}}{b^2q}{q,q^{n+1}},
		\end{align*}
		where we used the transformation \cite[(III.8)]{GR} in the second step. Since for $z\in \Z_{\geq 0}$,
		\[
		\rphis{2}{1}{q^{-z},ab^2q^{z}}{b^2q}{q,q^{n+1}} = \sum_{j=0}^z \frac{(q^{-z},ab^2q^{z};q)_j}{(b^2q,q;q)_j}q^{j(n+1)},
		\]
		we can rearrange the summations to obtain that the expression
		\[
		\sum_{n=0}^\infty (bcd)^n\frac{(a;q)_n}{(q;q)_n}\rphis{3}{2}{q^{n}, q^{x}, q^{-x}/ab^2}{1/a,0}{q^{-1},q^{-1}} \rphis{3}{2}{q^{n}, q^{y}, q^{-y}/ac^2}{1/a,0}{q^{-1},q^{-1}}
		\]
		is equal to
		\begin{align}
			\sum_{j_1=0}^{x} A(x,j_1,a,b) \sum_{j_2=0}^{y} A(y,j_2,a,c) \sum_{n=0}^\infty (bcd)^nq^{n(j_1+j_2)}\frac{(a;q)_n}{(q;q)_n},\label{eq:rearrangesums}
		\end{align}
		where
		\[
		A(z,j,a,b)= (-1/b^2)^zq^{-\half  z(z+1)}\frac{(b^2q;q)_z}{(a;q)_z}\frac{(q^{-z},ab^2q^{z};q)_j}{(b^2q,q;q)_j}q^{j}.
		\]
		If we now use the $q$-binomial theorem (\cite[(II.4)]{GR} if $a=q^{-N}$ with $N\in\Z_{\geq 0}$ or \cite[(II.3)]{GR} if \\ $|bcd|<1$ and $0<q<1$) we get
		\begin{align*}
			\sum_{n=0}^\infty(bcd)^nq^{n(j_1+j_2)}\frac{(a;q)_n}{(q;q)_n} &= \frac{(abcdq^{j_1+j_2};q)_\infty}{(bcdq^{j_1+j_2};q)_\infty},
		\end{align*}
		which can be written as
		\begin{align*}
			\frac{(bcd;q)_{j_1}(bcdq^{j_1};q)_{j_2}}{(abcd;q)_{j_1}(abcdq^{j_1};q)_{j_2}}\frac{(abcd;q)_\infty}{(bcd;q)_\infty},
		\end{align*}
		where we used four times that
		\[
		(aq^{j};q)_\infty = \frac{(a;q)_\infty}{(a;q)_j}
		\]
		for any $j\in \Z_{\geq 0}$. Therefore, \eqref{eq:rearrangesums} is equal to
		\begin{align}
			\begin{split}\frac{(abcd;q)_\infty}{(bcd;q)_\infty}&\sum_{j_1=0}^{x} A(x,j_1,a,b)\frac{(bcd;q)_{j_1}}{(abcd;q)_{j_1}}\sum_{j_2=0}^{y} A(y,j_2,a,c)\frac{(bcdq^{j_1};q)_{j_2}}{(abcdq^{j_1};q)_{j_2}}.\end{split}\label{eq:rearrangesums2}
		\end{align}
		Note that the second sum can be written as
		\begin{align*}
			(-1/c^2)^yq^{-\half y(y+1)}\frac{(c^2q;q)_y}{(a;q)_y}\rphis{3}{2}{q^{-y},ab^2q^y,bcdq^{j_1} }{c^2q,abcdq^{j_1}}{q;q}.
		\end{align*}
		Using the $q$-Saalsch\"utz formula \cite[(II.12)]{GR} for a balanced $\rphisempty{3}{2}$, we obtain\footnote{In \cite[(II,12)]{GR}, replace $(n,a,b,c)$ by $(y, ac^2q^y, bcdq^{j_1}, abcdq^{j_1})$.} 
		\begin{align*}
			\rphis{3}{2}{q^{-y},ab^2q^y,bcdq^{j_1} }{c^2q,abcdq^{j_1}}{q;q} = \frac{(bdq^{j_1-y}/c,a;q)_{y}}{(abcd  q^{j_1},q^{-y}/c^2;q)_{y}}.
		\end{align*}	
		Using
		\[
		(\alpha q^{j_1};q)_y = (\alpha;q)_y\frac{(\alpha q^{y};q)_{j_1}}{(\alpha;q)_{j_1}},
		\]
		finishes the proof.
	\end{proof}	
	\begin{remark}\*
		\begin{itemize}
			\item The $\rphisempty{4}{3}$-part in the lemma above is a rational function in the following way. Because $q^{-x}$ appears as a numerator parameter of the $\rphisempty{4}{3}$-series, the series terminates after $x$ terms. Therefore, it is of the form
			\[
				\sum_{n=0}^x \frac{1}{p_n\big(\lambda(y)\big)},
			\]
			where $p_n$ is a polynomial of degree $n$ in the variable $\lambda(y)=acq^y-q^{-y}$. Because of the symmetry
			\[
				(x,y,a,b,c,d) \leftrightarrow (y,x,a,c,b,d),
			\]
			it can also be seen as a rational function in the variable `$abq^x-q^{-x}$'.
			\item This summation formula can be interpreted as a Poisson kernel for the $q^{-1}$-Al-Salam--Chihara polynomials. The Poisson kernel for $q$-Al-Salam--Chihara polynomials with continuous variables can be found in \cite{ARS}.
		\end{itemize}
	\end{remark}	
\section{The quantum algebra $\U_q(\su_2)$ and rational functions of $q$-Racah type} \label{sec:su2rat}
In this section, we introduce the quantum algebra $\U_q(\mathfrak{su}_2)$, twisted primitive elements and show its well-known relation with the $q$-Krawtchouk polynomials. Using this, we introduce and solve a GEVP and an EVP and show, using the summation formula from Lemma \ref{lem:summation3phi2to4phi3}, that the overlap coefficients can be seen as a rational $\rphisempty{4}{3}$-function of $q$-Racah type. We will prove (bi)orthogonality as well as a recurrence relation for this function. Afterwards, we will extend these results to a multivariate setting, using the co-algebra structure of $\U_q(\mathfrak{sl}_2)$. Throughout this section, we take $q>0$ and $q \neq 1$.
\subsection{The algebra $\U_q(\mathfrak{su}_2)$ and twisted primitive elements}
The quantum algebra $\U_q(\mathfrak{su}_2)$ is the algebra $\U_q(\mathfrak{sl}_2)$ equipped with the $*$-structure which comes from the compact Lie algebra $\mathfrak{su}_2$. This is the anti-linear involution defined on the generators by
\[
	K^*=K, \quad E^*=F, \quad F^* = E, \quad (K^{-1})^* = K^{-1}.
\]
Let $u,s\in\C$, then we define the following elements in $\U_q(\mathfrak{su}_2)$, 
\begin{align*}
	X_{u,s} =& q^{u+\frac12} EK + q^{-u-\frac12}FK + [s]_q,\\
	\widetilde{X}_{u,s}=&q^{-u-\frac12} EK^{-1} + q^{u+\frac12}FK^{-1} + [s]_qK^{-2},
\end{align*}
which are almost twisted primitive elements. In particular, we have
\begin{align}
	\begin{split}\Delta(\widetilde{X}_{u,s})& = 1 \otimes (\widetilde{X}_{u,s}-[s]_qK^{-2}) + \widetilde{X}_{u,s}\otimes K^{-2}\\
	&=1 \otimes \widetilde{X}_{u,0} + \widetilde{X}_{u,s}\otimes K^{-2}. \end{split}\label{eq:DeltatildeX}
\end{align} 
That is, $\widetilde{X}_{u,s}$ generates a left co-ideal subalgebra. Note that if $s$ is real and $u$ is on the imaginary axis, then $X_{u,s}$ and $\widetilde{X}_{u,s}$ are self-adjoint in $\U_q(\mathfrak{su}_2)$. Later, we will need that we can write $X_{u,s}$ as a polynomial in $K^2$ and $\widetilde{X}_{v,t}$, similar to \cite[Lemma 4.3]{Gr21}.
\begin{lemma} \label{lem:XusinK2Xvt}
	Let $u,v,s,t\in\C$, then we have
	\begin{align*}
		X_{u,s}=\frac{q^{u+v}\big[K^2,\widetilde{X}_{v,t}\big]_q + q^{-u-v} \big[\widetilde{X}_{v,t},K^2\big]_q }{q^2-q^{-2}}- [t]_q\{u+v\}_q+[s]_q.
	\end{align*}
\end{lemma}
\begin{proof}
	The idea is to first write both $EK$ and $FK$ in terms of $\widetilde{X}_{v,t}$, $K^2$, and a constant. Then we take a linear combination of these expressions (and a constant) to construct $X_{u,s}$. Multiplying $\widetilde{X}_{v,t}$ on the left and right by $K^2$ gives
	\begin{align*}
		\widetilde{X}_{v,t}K^2 &= q^{-v-\half}EK + q^{v+\half}FK + [t]_q,\\
		K^2\widetilde{X}_{v,t} &= q^{-v+\frac{3}{2}}EK + q^{v-\frac{3}{2}}FK + [t]_q.
	\end{align*}
	Therefore, the $q$-commutator of $K^2$ and $\widetilde{X}_{v,t}$ cancels the `$FK$' term,
	\begin{align*}
		\big[K^2,\widetilde{X}_{v,t}\big]_q &= q^{-v+\half}(q^2-q^{-2})EK + (q-q^{-1})[t]_q.
	\end{align*}
	Similarly,
	\begin{align*}
		\big[\widetilde{X}_{v,t},K^2\big]_q &= q^{v-\half}(q^2-q^{-2})FK + (q-q^{-1})[t]_q.
	\end{align*}
	Consequently, one can write both $EK$ and $FK$ in terms of $\widetilde{X}_{v,t}$, $K^2$, and a constant. Taking a suitable linear combination of these expressions and a constant proves the theorem.
\end{proof}
\subsection{A representation of $\U_q(\mathfrak{su}_2)$ and $q$-Krawtchouk polynomials} Important functions in this paper are the (dual) $q$-Krawtchouk polynomials. These are defined by
\[
	K_n(x;c,N;q)= \rphis{3}{2}{q^{-n}, q^{-x}, cq^{x-N} }{q^{-N}, 0 }{q,q}.
\]
Note that $K_n(x)$ are polynomials in $cq^{x-N}+q^{-x}$ of degree $n$, in which case they are called dual $q$-Krawtchouk polynomials. However, they are also polynomials in $q^{-n}$ of degree $x$, in which case they are referred to as $q$-Krawtchouk polynomials. In the remaining of this paper, we will refer to both interpretations of $K_n(x)$ just as $q$-Krawtchouk polynomials. We will work with $q$-Krawtchouk polynomials which are normalized slightly different and have base $q^{-2}$, 
\begin{align*}
	k_{u,s}(n,x)=k_{u,s}(n,x;N;q)&= (-1)^nq^{n(s-u-\frac12 N +\frac12)} K_n(x;-q^{-2s},N;q^{-2})\\
	&=(-1)^nq^{n(s-u-\frac12 N +\frac12)}\rphis{3}{2}{q^{2n}, q^{2x}, -q^{-2x-2s+2N} }{q^{2N}, 0 }{q^{-2},q^{-2}}.
\end{align*}
Note that 
\begin{align}
	k_{u,s}(n,x)= q^{-un}k_{0,s}. \label{eq:kusk0s}
\end{align}
If $s\in \R$, the $q^{-1}$-Krawtchouk polynomials have a finite orthogonality relation in $n$ as well as $x$,
\begin{align}
	\sum_{x=0}^N k_{0,s}(n,x)k_{0,s}(n',x)W(x,s;q^{-1}) = \frac{\delta_{n,n'}}{w(n)},\label{eq:KrawOrthx}\\
	\sum_{n=0}^N k_{0,s}(n,x)k_{0,s}(n,x')w(n) = \frac{\delta_{x,x'}}{W(x,s;q^{-1})},\label{eq:KrawOrthn}
\end{align}
where the weight functions are given by
\begin{align*}
	w(n)&=w(n,N;q) = q^{n(n-N)} \qbinom{q^2}{N}{n}=w(n,N;q^{-1}),\\
	W(x,s;q)&=W(x,s,N;q) =\frac{ 1+q^{4x+2s-2N}}{1+q^{2s-2N}} \frac{ (-q^{2s-2N};q^2)_x }{(-q^{2s+2};q^2)_x } \frac{q^{-x(2s+1+x-2N)}}{(-q^{-2s};q^2)_N} \qbinom{q^2}{N}{x}.
\end{align*}
The first can be found in \cite[(14.17.2)]{KLS}, the second is its dual\footnote{If a finite square matrix $U$ satisfies $UU^\mathrm{T}=I$, it also satisfies the dual orthogonality $U^\mathrm{T}U=I$.} orthogonality. Next, for this section, we define $H_{N}$ to be the $(N+1)$-dimensional Hilbert space of (continuous) functions $f\colon\{0,1,...,N\}\to\C$ with inner product induced by the orthogonality measure $w$ of the $q^{-1}$-Krawtchouk polynomials, 
\[
\langle f,g \rangle_{H_N} = \sum_{n=0}^{N} f(n) \overline{g(n)} w(n).
\]
Let $B(H_{N})$ be the space of linear operators on $H_{N}$ and $\pi_{N}\colon\U_q(\mathfrak{su}_2)\to B(H_{N})$ the $*$-representation defined by
\begin{equation} \label{eq:representation}
	\begin{split}
		[\pi_N(K)f](n ) &= q^{n-\frac12 N} f(n), \\
		[\pi_N(E)f](n) &= [n]_q f(n-1), \\
		[\pi_N(F) f](n) & = [N-n]_q f(n+1),\\
		[\pi_N(K^{-1}) f](n) &= q^{\frac12N-n}f(n).
	\end{split}
\end{equation}
One can easily verify that this is a $*$-representation, i.e. 
\[
	\langle \pi_N(X)f,g\rangle = \langle f, \pi_N(X^*)g\rangle
\] 
for all $X\in \U_q(\mathfrak{su}_2)$ and $f,g \in H_N$, by checking this for the generators $K,K^{-1},E$ and $F$ on the basis of delta functions.\\

It is well known (\cite{Koo}) that the $q^{-1}$-Krawtchouk polynomials $k_{u,s}(\cdot,x): \{0,...,N\}\to \C$ are eigenfunctions of $\pi_N(\widetilde{X}_{u,s})$,
\begin{align}
	[\pi_N(\widetilde{X}_{u,s})k_{u,s}(\cdot,x)](n)=[2x-N+s]_qk_{u,s}(n,x). \label{eq:widetildeXEV}
\end{align}
This follows from matching the explicit action $\pi_N(\widetilde{X}_{u,s})$ with the three-term recurrence relation of the dual $q$-Krawtchouk polynomials \cite[(14.17.3)]{KLS}.

\subsection{Rational functions of $q$-Racah type as overlap coefficient}
The elements $X_{0,s}+[s]_q(K^2-1)$ and $\widetilde{X}_{v,t}$ generate an Askey-Wilson algebra, see e.g. \cite{GranZhed} or \cite{GroeneveltWagenaar}. Consequently, overlap coefficients between their eigenfunctions are $q$-Racah polynomials. In other words, the $q$-Racah polynomials appear as overlap coefficients of the solutions of the following eigenvalue problems (EVPs),
\begin{align}
	&\pi_N^{}\big(X_{0,s}^{}+[s]_q(K^2-1)\big)f=\lambda f,\label{eq:EVPXs}\\
	&\pi_N^{}(\widetilde{X}_{v,t}^{})g=\mu g. \label{eq:EVPXt}
\end{align}
In this paper, we will be interested in the overlap coefficients between solutions of the EVP \eqref{eq:EVPXt} and the GEVP
\begin{align}
	\pi^{}_N(X_{0,s})f=\lambda \pi_N^{}(K^2) f. \label{eq:GEVPUq}
\end{align}
We can rewrite this GEVP to an EVP. Indeed, since $K^2$ is invertible, \eqref{eq:GEVPUq} is equivalent to
\begin{align*}
	\pi_N(K^{-2}X_{0,s})f=\lambda f.
\end{align*}
Using the commutation relations for $\U_q(\mathfrak{sl}_2)$, we obtain
\begin{align*}
	K^{-2}X_{0,s}^{} = \widetilde{X}_{1,s}^{}.
\end{align*}
Since $k_{1,s}(\cdot,x)$ are the solutions of the EVP
\begin{align}
	\pi_N^{}(\widetilde{X}_{1,s})f=\lambda f, \label{eq:GEVPrewritten}
\end{align}
we know that $k_{1,s}(\cdot,x)$ also solves the GEVP \eqref{eq:GEVPUq}. Let $R_\mathrm{r}(x,y)$ be the overlap coefficients between the solutions of the EVP \eqref{eq:EVPXt} and the GEVP \eqref{eq:GEVPUq},
\begin{align}
	R_\mathrm{r}(x,y)=R_\mathrm{r}(x,y;s,t,v,N;q)=\langle k_{1,s}(\cdot,x), k_{v,t}(\cdot,y)  \rangle_{H_N^{}}^{}. \label{eq:overlapcoef}
\end{align}  
Note that these are overlap coefficients between two (different) $q^{-1}$-Krawtchouk polynomials. The subscript `$\mathrm{r}$' is there to emphasize it is a rational function, as will be proved in the next proposition using the summation formula from Lemma \ref{lem:summation3phi2to4phi3}.
\begin{proposition}\label{prop:overlaprational}
	$R_\mathrm{r}(x,y)$ is a rational function of $q$-Racah type,
	\[
		R_\mathrm{r}(x,y)= 	c_1\rphis{4}{3}{q^{-2x},-q^{2x+2s-2N},-q^{s+t-v+1},q^{s-t-v+1}}{-q^{2s+2},q^{-2y+s-t-v+1},-q^{2y+s+t-2N-v+1}}{q^2;q^2},
	\]
	where 
	\[
		c_1 = (-q^{-2N+s+t-v+1};q^2)_N \frac{(-q^{-2x-2s};q^2)_{x}}{(q^{-2N};q^2)_{x}}\frac{(q^{-2y+s-t-v+1};q^2)_{y}}{(-q^{s+t-2N-v+1};q^2)_{y}}.
	\]
\end{proposition}
\begin{proof}
	Writing out the inner product in the definition of $P_\mathrm{r}(x,y)$ gives
	\begin{align*}
		R_\mathrm{r}(x,y)&= \sum_{n=0}^N q^{n(s+t-v-N+n)}\qbinom{q^2}{N}{n} \rphis{3}{2}{q^{2n}, q^{2x}, -q^{-2x-2s+2N} }{q^{2N}, 0 }{q^{-2},q^{-2}}\\
		&\hspace{5cm}\times \rphis{3}{2}{q^{2n}, q^{2y}, -q^{-2y-2t+2N} }{q^{2N}, 0 }{q^{-2},q^{-2}}.
	\end{align*}
	Since by \cite[(I.12)]{GR}, we have
	\begin{align*}
		\frac{(q^2;q^2)_N}{(q^2;q^2)_{N-n}}=(-1)^nq^{n(N+1-n)}(q^{-2N};q^2)_n,
	\end{align*} 
	the previous expression is equal to
	\begin{align*}
		&\sum_{n=0}^N (-1)^nq^{n(s+t-v+1)}\frac{(q^{-2N};q^2)_n}{(q^2;q^2)_n} \rphis{3}{2}{q^{2n}, q^{2x}, -q^{-2x-2s+2N} }{q^{2N}, 0 }{q^{-2},q^{-2}}\\
		&\hspace{5cm}\times \rphis{3}{2}{q^{2n}, q^{2y}, -q^{-2y-2t+2N} }{q^{2N}, 0 }{q^{-2},q^{-2}}.
	\end{align*}
	Therefore, the required identity follows from Lemma \ref{lem:summation3phi2to4phi3} with $q$ replaced by $q^2$ and taking $(a,b,c,d)$ to be \[(q^{-2N},iq^s,iq^t,q^{-v+1}). \qedhere\]
\end{proof}
\begin{remark} \label{rem:R_r=delta}
	In the special case when $s=t$, the summation reduces to the orthogonality \eqref{eq:KrawOrthn} of the $q^{-1}$-Krawtchouk polynomials. Thus in that case, we have $R_\mathrm{r}(x,y)=0$ when $x\neq y$. 
\end{remark}
\begin{remark} The function $R_\mathrm{r}(x,y)$ is related to some other functions.
	\begin{itemize}
		\item It is related to the rational $q$-Racah functions \cite[(57)]{CTVZ} as follows. The $\rphisempty{4}{3}$-part of that function is given by
		\[
			\rphis{4}{3}{q^{-n},bq^{n+1},a/c,ad}{aq^{x+1},aq^{-x}/c,bdq}{q,q}.
		\]
		Replacing $q$ by $q^2$ and taking 
		\begin{align*}
			x=y, n=x, a=-q^{s +t-v-2N-1}, b=-q^{2s-2N-2}, c=-q^{2t-2N-2}, d=q^{2N+2},
		\end{align*}
		gives back the $\rphisempty{4}{3}$-part of $R_\mathrm{r}$.
		\item 	It is related to the rational functions of Hahn-type $\mathcal{U}_n(x;\alpha,\beta,N)$ from \cite[(1.4)]{VinZhe} as follows. If we apply Sear's transformation \cite[(III.15)]{GR} for a terminating balanced $\rphisempty{4}{3}$ to $R_\mathrm{r}(x,y)$, we obtain, only looking at the $\rphisempty{4}{3}$-part,
		\begin{align*}
			\rphis{4}{3}{q^{-2x},-q^{2x+2s-2N},-q^{-2y-2t},q^{-2y}}{q^{-2N},q^{-2y+s-t-v-1},q^{-2y+s-t+v+1}}{q^2,q^2}.
		\end{align*}
		If we now replace $(q^{t},q^s,v)$ by $(iq^{t},iq^s, 2v-s-t-1)$ and let $q\to1$, we get
		\begin{align*}
			\rFs{4}{3}{-x,x+s-N,-y-t,-y}{-N,-y+ s-v,-y- t +v}{1},
		\end{align*}
		where $_{r}F_{s}$ is the regular hypergeometric series as in \cite{KLS}. Now let $t\to\infty$ to obtain
		\begin{align*}
			\rFs{3}{2}{-x,\ x+s-N,\ -y}{-N,-y+ s-v}{1},
		\end{align*}
		which is exactly \cite[(1.4)]{VinZhe} with $x=y, n=x,\alpha=s-v$ and $\beta=s$.\\
		\item It is a special case of Wilson's biorthogonal rational functions $r_n(z)$ in \cite[(2.10)]{Wi}. These $r_n(z)$ are defined by a product of $q$-shifted factorials and a very-well poised $\rphisempty{10}{9}$. The latter is given by
		 \begin{align*}
		 	\rphis{10}{9}{a/e, q\sqrt{a/e}, -q\sqrt{a/e}, a/z, az, q/be, q/ce, q/de, q^n/ef,q^{-n}}{\sqrt{a/e}, - \sqrt{a/e}, qz/e, q/ez,ab,ac,ad,q^{1-n}af,q^{n+1}a/e}{q,q},
		 \end{align*}
		 where $abcdef=q$. Substituting $f=q/abcde$, $z=aq^m$ and replacing $c$ by $c/b$, we get
 		 \begin{align*}
		 	\rphis{10}{9}{a/e, q\sqrt{a/e}, -q\sqrt{a/e}, q^{-m}, a^2q^m, q/be, qb/ce, q/de, acdq^{n-1},q^{-n}}{\sqrt{a/e}, - \sqrt{a/e}, qaq^m/e, q/ez,ab,ac/b,ad,q^{2-n}/cde,q^{n+1}a/e}{q,q}.
		 \end{align*}
		 Letting $b\to0$ gives
		 \begin{align*}
		 	\rphis{8}{7}{a/e, q\sqrt{a/e}, -q\sqrt{a/e}, q^{-m}, a^2q^m, q/de, acdq^{n-1},q^{-n}}{\sqrt{a/e}, - \sqrt{a/e}, aq^{m+1}/e, q^{1-m}/ae,ad,q^{2-n}/cde,q^{n+1}a/e}{q,q^2/ace}.
		 \end{align*}
		 Applying Watson's transformation formula \cite[(III.18)]{GR}, we obtain a factor times
		 \begin{align*}
		 	\rphis{4}{3}{q/ae,q/de,acdq^{n-1},q^{-n}}{aq^{m+1}/e,q^{1-m}/ae,c}{q,q}.
		 \end{align*}
		 Taking $q^2$ instead of $q$, $n=x$, $m=y$ and doing the change of parameters
		 \begin{align*}
		 	a=-iq^{t-N},\ c=-q^{2s+2},\ d=iq^{-N-t},\ e=iq^{1-s+N+v},
		 \end{align*}
		 we exactly get the $\rphisempty{4}{3}$-part of $R_\mathrm{r}(x,y)$.
		 \item In \cite[Appendix 2]{BR} a classification of biorthogonal rational functions of $q$-hypergeometric type is obtained by considering limits of elliptic hypergeometric biorthogonal functions. In this classification the $q$-Racah type rational functions we consider correspond to the case denoted by $22v2$.
	\end{itemize}
\end{remark}
\subsection{Biorthogonality and $q$-difference equations of $R_\mathrm{r}$}
	This algebraic interpretation of the rational function $R_\mathrm{r}$ can be used to derive biorthogonality relations and $q$-difference equations of this function. Let us start with the biorthogonality relations, which follow from the orthogonality relations of the $q^{-1}$-Krawtchouk polynomials \eqref{eq:KrawOrthn} and \eqref{eq:KrawOrthx}.
	\begin{proposition}\label{prop:biorthraqra}
		For $s,t\in \R$, we have the following biorthogonality relations for $R_\mathrm{r}$,
		\begin{align}
			&\sum_{x=0}^N R_\mathrm{r}(x,y;s,t,v,N)\overline{R_\mathrm{r}(x,y';s,t,-\overline{v}-2,N)}W(x,s;q^{-1})=\frac{\delta_{y,y'}}{W(y,t;q^{-1})}, \label{eq:RatRacOrthx}\\
			&\sum_{y=0}^N R_\mathrm{r}(x,y;s,t,v,N)\overline{R_\mathrm{r}(x',y;s,t,-\overline{v}-2,N)}W(y,t;q^{-1})=\frac{\delta_{x,x'}}{W(x,s;q^{-1})}.\label{eq:RatRacOrthy}
		\end{align}
	\end{proposition}
	\begin{proof}
		This follows from the orthogonality relations \eqref{eq:KrawOrthx} and \eqref{eq:KrawOrthn} of the $q^{-1}$-Krawtchouk polynomials. For this proof, let $\mathcal{H}_N^s$ be the $(N+1)$-dimensional Hilbert space with inner product given by
			\begin{align*}
				\langle f,g\rangle_{\mathcal{H}_N^s} = \sum_{x=0}^{N}f(x)\overline{g(x)}W(x,s;q^{-1}).
			\end{align*}
			Let $U_s\colon H_N\to\mathcal{H}_N^s$ be the linear operator defined by
			\[
			(U_s f)(x)= \langle f, k_{0,s}(\cdot,x)\rangle_{H_N} = \sum_{n=0}^N f(n)k_{0,s}(n,x)w(n).
			\]
			Note that applying $U_s$ to a delta-function\footnote{That is, $\delta_n(m)=1$ if $n=m$ and $0$ otherwise.} gives $$ (U_s \delta_n)(x)=k_{0,s}(n,x)w(n).$$
			Therefore, $U_s$ is unitary since it sends the orthogonal basis of rescaled delta-functions $\{\delta_n/w\}_{n=0}^N$ to the orthogonal basis $\{k_{0,s}(\cdot,n)\}_{n=0}^N$, while preserving its norm,
			\[
			\|U_s (\delta_n/w)\|^2_{\mathcal{H}_N^s} = \sum_{x=0}^N k_{0,s}(n,x)^2W(x,s;q^{-1}) = \frac{1}{w(n)} = \|\delta_n/w\|^2_{H_N},
			\]
			where we used \eqref{eq:KrawOrthx}. Note that applying $U_s$ to $k_{v+1,t}(\cdot,y)$ gives back our rational function $R_\mathrm{r}(x,y)$. Therefore, we obtain the first (bi)orthogonality relation,
			\begin{align*}
				\langle R_\mathrm{r}(x,y;s,t,v,N), R_\mathrm{r}(x,y';s,t,-\overline{v}-2,N) \rangle_{\mathcal{H}_N^s} &= \langle U_s k_{v+1,t}(\cdot,y), U_s k_{-\overline{v}-1,t}(\cdot,y') \rangle_{\mathcal{H}_N^s}\\
				&= \langle k_{v+1,t}(\cdot,y), k_{-\overline{v}-1,t}(\cdot,y') \rangle_{H_N} \\
				&= \sum_{n=0}^N k_{v+1,t}(n,y)\overline{k_{-\overline{v}-1,t}(n,y')} w(n) \\
				&= \sum_{n=0}^N k_{0,t}(n,y)k_{0,t}(n,y') w(n)\\
				&= \frac{\delta_{y,y'}}{W(y,t;q^{-1})},
			\end{align*} 
			where we used \eqref{eq:kusk0s} and \eqref{eq:KrawOrthn}. The second biorthogonality relation \eqref{eq:RatRacOrthy} follows directly from the first since we have the symmetry $R_\mathrm{r}(y,x;t,s)=R_\mathrm{r}(x,y;s,t)$.
	\end{proof}
	\begin{remark}\label{rem:biorth=orth}
		If $v=-\overline{v}-2$ we have orthogonality instead of biorthogonality. This is the case when $\text{Re}(v)=-1$.
	\end{remark}
	\begin{remark}
		By the theory of (G)EVPs (see e.g.~(2.20) in \cite{TVZ}), we have the biorthogonality relations,
		\begin{align*}
			&\sum_{x=0}^N R_\mathrm{r}(x,y) \overline{\widetilde{R}_\mathrm{r}(x,y')} = \kappa_y \delta_{y,y'},\\
			&\sum_{y=0}^N R_\mathrm{r}(x,y) \overline{\widetilde{R}_\mathrm{r}(x',y)} = \kappa'_x \delta_{x,x'},
		\end{align*}
		for some constants $\kappa_y$ and $\kappa'_x$. Here, $\widetilde{R}_\mathrm{r}(x,y')$ is the overlap coefficient
		\[
			\widetilde{R}_\mathrm{r}(x,y) = \langle k_{v,t}^*(\cdot,y), \pi_N(K^2)k_{1,s}(\cdot,x)\rangle_{H_N},
		\] 
		where $k_{v,t}^*(n,y)$ is the solution of the adjoint EVP of \eqref{eq:EVPXt},
		\[
			\pi_N^{}(\widetilde{X}^*_{v,t})g=\mu g.
		\]
		This biorthogonality is equivalent to the one of the previous proposition. Indeed,
		since $$\widetilde{X}^*_{v,t} = \widetilde{X}_{-\overline{v},t},$$ we have $k_{v,t}^*(n,y)= k_{-\overline{v},t}(n,y)$. Moreover, by the definition of $\pi_N(K^2)$ and \eqref{eq:kusk0s},
		\[
			\pi_N(K^2)k_{1,s}(n,x)=q^{-N} k_{-1,s}(n,x).
		\]
		Therefore,
		\[
			\widetilde{R}_\mathrm{r}(x,y) = q^{-N}R_\mathrm{r}(x,y';s,t,-v-2,N),
		\]
		showing the equivalence of the biorthogonality relations.
	\end{remark}
	Let us now deduce a recurrence relation for $R_\mathrm{r}(x,y)$. For this, we need two ingredients. First, using that $k_{1,s}$ solves the GEVP \eqref{eq:GEVPUq} and that both $K^2$ and $X_{0,s}$ are self-adjoint in $\U_q(\mathfrak{su}_2)$, we have 
	\begin{align}
		[2x-N+s]_q \Big\langle k_{1,s}(\cdot,x),\pi_N(K^2)k_{v,t}(\cdot,y) \Big\rangle_{H_N} &=  \Big\langle [2x-N+s]_q\nonumber  \pi_N(K^2)k_{1,s}(\cdot,x),k_{v,t}(\cdot,y) \Big\rangle_{H_N}\\
			&= \Big\langle \pi_N(X_{0,s})k_{1,s}(\cdot,x),k_{v,t}(\cdot,y) \Big\rangle_{H_N}\label{eq:CorGEVP}\\
			&= \Big\langle k_{1,s}(\cdot,x),\pi_N(X_{0,s})k_{v,t}(\cdot,y) \Big\rangle_{H_N}.  \nonumber
	\end{align}
	Secondly, we can transfer the action of $\pi_N(K^2)$ and $\pi_N(X_{0,s})$ on the `$n$' variable of $k_{v,t}$ to the `$y$' variable, which is the content of the following lemma. 
	\begin{lemma}\label{lem:threetermKra}
		The action of the operators $\pi_N(K^2)$ and $\pi_N(X_{0,s})$ on the $n$ variable of $k_{v,t}(n,y)$ can be transferred to a $3$-diagonal action in the other variable,
		\begin{align}
			\left[\pi_N(K^2)k_{v,t}(\cdot,y)\right](n) &= a_{-1}^{} k_{v,t}(n,y-1) + a_{0}^{}k_{v,t}(n,y)+ a_{1}^{} k_{v,t}(n,y+1),\label{eq:3termK2}\\
			 \left[\pi_N(X_{0,s})k_{v,t}(\cdot,y)\right](n) &=  b_{-1}^{} k_{v,t}(n,y-1) + (b_{0}^{}+[s]_q) k_{v,t}(n,y)+ b_{1}^{} k_{v,t}(n,y+1),\label{eq:3termYs}
		\end{align}
		where
		\begin{align*}
			a_{-1}^{}(y,t) &= -\frac{q^{-4y-2t+3N+2}(1-q^{-2y})(1+q^{-2y-2t})}{(1+q^{-4y-2t+2N+2})(1+q^{-4y-2t+2N})}, \\
			a_{0}^{}(y,t) &= -(a_{-1}(y)+a_{1}(y)-1),\\
			a_{1}^{}(y,t) &= \frac{q^{-N}(1-q^{-2y+2N})(1+q^{-2y-2t+2N})}{(1+q^{-4y-2t+2N})(1+q^{-4y-2t+2N-2})},
		\end{align*}
		and
		\begin{align*}
			b_{-1}(y,t)&= a_{-1}(y,t)[2y-N+t+v-1]_q ,\\
			b_{0}(y,t)&= a_0(y,t)[2y-N+t]_q\{v\}_q- [t]_q\{v\}_q,\\
			b_{1}(y,t)&= a_1(y,t)[2y-N+t-v+1]_q .
		\end{align*}
	\end{lemma}
	\begin{proof}
		Since $\pi_N^{}(K^2)$ is just multiplication by $q^{2n-N}$, the first equation \eqref{eq:3termK2} follows from the 3-term recurrence relations of the $q$-Krawtchouk polynomials (\cite[(14.17.5)]{KLS}, where we replace $q$ by $q^{-2}$ and multiply both sides by $q^{-N}$.\\
		\indent The other equation follows from Lemma \ref{lem:XusinK2Xvt}, which shows that we can write $X_{u,s}$ in terms of $K^2$ and $\widetilde{X}_{v,t}$, combined with \eqref{eq:3termK2} and the eigenvalue equation \eqref{eq:widetildeXEV} for $\widetilde{X}_{v,t}$:
		\begin{align}
			\big[\pi_N(\widetilde{X}_{v,t})k_{v,t}(\cdot,y)\big](n)=[2y-N+t]k_{v,t}(\cdot,y). \label{eq:Xtuev}
		\end{align}
		From Lemma \ref{lem:XusinK2Xvt} with $u=0$ we get
		\begin{align}
			X_{0,s}=\frac{q^{v}\big[K^2,\widetilde{X}_{v,t}\big]_q + q^{-v} \big[\widetilde{X}_{v,t},K^2\big]_q }{q^2-q^{-2}}- [t]_q\{v\}_q+[s]_q,\label{eq:Xs1K2Xtu}
		\end{align}
		If we now use the expressions \eqref{eq:3termK2} and \eqref{eq:Xtuev} for $K^2$ and $\widetilde{X}_{v,t}$ respectively, we obtain
		\begin{align*}
			\big[\pi_N(X_{0,s})k_{v,t}(\cdot,x)\big](n) =  Aa_{-1}(y)k_{v,t}(n,y-1)+\bigg[a_{0}(y)B +C\bigg]k_{v,t}(n,y)+a_{1}(y)Dk_{v,t}(n,y+1)
		\end{align*}
		where
		\begin{align*}
			A&=\frac{[2y-N+t]_q)(q^{v+1}-q^{-v-1})+[2y-2-N+t]_q(q^{-v+1}-q^{v-1}) }{q^2-q^{-2}} \\
			&= [2y-N+t+v-1]_q, \\
			B&= \frac{(q-q^{-1})(q^v+q^{-v})[2y-N+t]_q}{q^2-q^{-2}} \\
			&= [2y-N+t]_q\{v\}_q ,\\
			C&= [s]_q- [t]_q\{v\}_q, \\
			D&=\frac{[2y-N+t]_q)(q^{v+1}-q^{-v-1})+[2y+2-N+t]_q(q^{-v+1}-q^{v-1}) }{q^2-q^{-2}}\\
			&= [2y-N+t-v+1]_q. \qedhere
		\end{align*}
	\end{proof}
 	Combining \eqref{eq:CorGEVP} and Lemma \ref{lem:threetermKra}, we obtain a GEVP for $R_\mathrm{r}(x,y)$.
	\begin{corollary}\label{cor:GEVPR(x,y)}
		The rational function $R_\mathrm{r}(x,y)$ satisfies
		\begin{align*}
			[2x-N+s]_q\Big(a_{-1}^{}(y,t) &R_\mathrm{r}(x,y-1) + a_{0}^{}(y,t) R_\mathrm{r}(x,y)+ a_{1}^{}(y,t) R_\mathrm{r}(x,y+1)\Big) \\
			&= b_{-1}(y,t) R_\mathrm{r}(x,y-1) + (b_{0}(y,t)+[s]_q) R_\mathrm{r}(x,y)+ b_{1}(y,t) R_\mathrm{r}(x,y+1).
		\end{align*}
	\end{corollary}
	\begin{remark}
		The rational function $R_\mathrm{r}(x,y)$ also satisfies an EVP. To see this, first note that
		\[
			R_\mathrm{r}(x,y)= \Big\langle k_{1,s}(\cdot,x),k_{v,t}(\cdot,y) \Big\rangle_{H_N} = \Big\langle k_{1+v,s}(\cdot,x),k_{0,t}(\cdot,y) \Big\rangle_{H_N}.
		\] 
		Recall that we can rewrite the GEVP \eqref{eq:GEVPUq} into the EVP \eqref{eq:GEVPrewritten}. Therefore,
		\begin{align}
			\begin{split} [2x-N+s]_q \Big\langle k_{1+v,s}(\cdot,x),k_{0,t}(\cdot,y) \Big\rangle_{H_N} &=  \Big\langle \pi_N(\widetilde{X}_{1+v,s})k_{1+v,s}(\cdot,x),k_{0,t}(\cdot,y) \Big\rangle_{H_N} \\
			&= \Big\langle k_{1+v,s}(\cdot,x),\pi_N(\widetilde{X}_{1+v,s})^*k_{0,t}(\cdot,y)  \Big\rangle_{H_N}. \end{split}\label{eq:remEVPR}
		\end{align}
		It is possible to transfer the action of $\pi_N(\widetilde{X}_{1+v,s})^*$ on the $n$-variable of $k_{0,t}$ to the $y$-variable. From the proof of Proposition \ref{prop:biorthraqra}, we have the unitary operator $U_t\colon H_N\to\mathcal{H}_N^t$.  This allows us to define a $*$-representation $\rho\colon \U_q(\mathfrak{su}_2)\to B(\mathcal{H}_N^s)$ by
		\[
			\rho(X)= U_t \pi_N (X) U_t^{-1},
		\]
		which satisfies
		\begin{align*}
			\big[\pi_N(X)k_{0,t}(n,\cdot)\big](y) =\big[\rho(X^*)k_{0,t}(n,\cdot)\big](y).
		\end{align*}
		To see this, observe that
		\begin{align*}
			w(n)\big[\pi_N(X)k_{0,t}(\cdot,y)\big](n) &= \langle \delta_n,\pi_N(X)k_{0,t}(\cdot,y) \rangle_{H_N} = \langle \pi_N(X^*)\delta_n,k_{0,t}(\cdot,y) \rangle_{H_N} = \big[U_t\pi(X^*)\delta_n\big](y) \\
			&= \big[\rho(X^*)U_t\delta_n\big](y)= w(n)\big[ \rho(X^*)k_{0,t}(n,\cdot)\big](y).
		\end{align*}
		Therefore, there exist constants $\{\alpha_y\}_{y=0}^N$ such that
		\[
			\pi_N(\widetilde{X}_{1+v,s})^*k_{0,t}(\cdot,y)(n) = \rho(\widetilde{X}_{1+v,s})k_{0,t}(n,\cdot)(y) = \sum_{y=0}^N \alpha_y k_{0,t}(n,y).
		\]
		Thus from \eqref{eq:remEVPR} we get
		\[
			[2x-N+s]_q R_\mathrm{r}(x,y) = \sum_{y=0}^N \alpha_y R_\mathrm{r}(x,y).
		\]
		Note that it is quite difficult to get an explicit expression for the coefficients $\alpha_y$ in this EVP. 
	\end{remark}
\subsection{Multivariate rational functions of $q$-Racah type}
	Define the $j$-th coproduct $\Delta^j\colon \U_q(\mathfrak{sl}_2)\to U_q(\mathfrak{sl}_2)^{\otimes (j+1)}$ recursively by
	\begin{align*}
		\Delta^j = (\Delta \otimes 1^{\otimes (j-1)}) \Delta^{j-1},\qquad j\geq 1,
	\end{align*}
	with the convention that $\Delta^0$ is the identity on $\U_q(\mathfrak{sl}_2)$. For $\vec{N}=(N_1,\ldots,N_M)\in(\Z_{\geq0})^M$, let $H_{\vec{N}}$ denote the algebraic tensor product $H_{N_1}\otimes \cdots \otimes H_{N_M}$. Then $H_{\vec{N}}$ is isomorphic with the Hilbert space of functions on $\{0,\ldots,N_1\}\times \cdots \times \{0,\ldots,N_M\}$, with inner product
	\[
		\langle f, g \rangle_{H_{\vec{N}}} = \sum_{n_1=0}^{N_1}\cdots \sum_{n_M=0}^{N_M} f(\vec{n})\overline{g(\vec{n})} w(\vec{n}),
	\]
	where $w(\vec{n})$ is the product of the univariate weight functions,
	\[
		w(\vec{n}) = \prod_{j=1}^M w(n_j).
	\]
	Define the tensor product representation $\pi_{\vec{N}}$ of $\U_q(\mathfrak{sl}_2)^{\otimes M}$ by
	\[
		\pi_{\vec{N}}(X_1\otimes \cdots \otimes X_M) = \pi_{N_1}(X_1)\otimes \cdots \otimes \pi_{N_M}(X_M).
	\]
	We will often look at elements in $\U_q(\mathfrak{sl}_2)^{\otimes M}$ that have the `$j-1$'-th coproduct of $X\in\U_q(\mathfrak{sl}_2)$ on the left or right `$j$' spots and ones on the other spots. To ease notation, define 
	\[
		\Delta^{j-1}_\mathrm{L}= \Delta^{j-1}(X)\otimes \underbrace{ 1 \otimes \cdots \otimes 1}_{M-j\ \text{times}} \qquad \text{and} \qquad \Delta^{j-1}_\mathrm{R} = \underbrace{ 1 \otimes \cdots \otimes 1}_{M-j\ \text{times}}\otimes\Delta^{j-1}(X).
	\]
	Now we can define the following multivariate analogue of the GEVP \eqref{eq:GEVPUq} and EVP \eqref{eq:EVPXt}. We want to find $f,g \in H_{\vec{N}}$ such that for all $j=1,\ldots,M,$
\begin{align}
	\pi_{\vec{N}}  \Big(\Delta^{j-1}_\mathrm{L}\big(X_{0,s}\big)\Big)f&=\lambda_j \pi_{\vec{N}}^{}\Big(\Delta^{j-1}_\mathrm{L}\big(K^2\big)\Big) f, \label{eq:GEVPUqmulti}\\
	\pi_{\vec{N}}\Big(\Delta^{j-1}_\mathrm{L}\big(\widetilde{X}_{v,t}\big)\Big)g&= \mu_j g. \label{eq:EVPXtmulti}
\end{align}
Again, we can rewrite \eqref{eq:GEVPUqmulti} to an EVP by multiplying both sides by $\pi_{\vec{N}}^{}\big(\Delta^{j-1}_\mathrm{L}(K^{-2})\big)$,
\begin{align}
	\pi_{\vec{N}}\Big(\Delta^{j-1}_\mathrm{L}(\widetilde{X}_{1,s})\Big)f&= \lambda_j f.\label{eq:GEVPrewrittenmulti}
\end{align}
Let us first solve these EVPs. It is well known \cite{KoeJeu} that a nested product of $q^{-1}$-Krawtchouk polynomials are eigenfunctions of $\pi_{\vec{N}}\big(\Delta^{j-1}_\mathrm{L}(\widetilde{X}_{v,t})\big)$. Let
\begin{align*}
	K_{v,t}\big(\vec{n},\vec{y}\big) = \prod_{j=1}^{M} k^{}_{v,h^{}_{j-1}(\vec{y},t)}(n_j,y_j),
\end{align*}
where the height function $h_j$ is defined recursively by,
\begin{align*}
	h_j&=h_{j-1}+2y_j - N_j,\qquad j \geq 1,\\
	h_0&= t,
\end{align*}
so that 
\begin{align*}
	h_{j}=h_j(\vec{y},t)=t + \sum_{i=1}^j 2y_i - N_i.
\end{align*} 
$K_{v,t}$ has a nested structure in the sense that $k^{}_{v,h^{}_{j-1}(\vec{y},t)}(n_j,y_j)$ depends on $y_1,\ldots,y_{j-1}$ via the height function $h^{}_{j-1}(\vec{y},t)$. The next proposition shows that $f=K_{1,s}\big(\vec{n},\vec{x}\big)$ and $g=K_{v,t}\big(\vec{n},\vec{y}\big)$ solve the EVP's \eqref{eq:GEVPrewrittenmulti} with $\lambda_j=[h_j(\vec{x},s)]_q$ and \eqref{eq:EVPXtmulti} with $\mu_j=[h_j(\vec{y},t)]_q$ respectively.
\begin{proposition}\label{prop:multivarkrawEV}
	$K_{v,t}\big(\vec{n},\vec{y}\big)$	are eigenfunctions of $\pi_{\vec{N}}\big(\Delta^{j-1}_\mathrm{L}(\widetilde{X}_{v,t})\big)$ for $j=1,...,M$,
	\begin{align}
		\big[\pi_{\vec{N}}\Big(\Delta^{j-1}_\mathrm{L}(\widetilde{X}_{v,t})\Big)K_{v,t}(\cdot,\vec{y})\big](\vec{n}) = [h_{j}(\vec{y},t)]_q K_{v,t}(\vec{n},\vec{y}). \label{eq:deltaXtuev}
	\end{align}
\end{proposition}
\begin{proof}
	We prove this using induction on $j$. Note that for any integer $j$ between $1$ and $M$ we have
	\[
	K_{v,t}(\vec{n},\vec{y}) = K_{v,t}((n_1,\ldots,n_j),(y_1,\ldots,y_j)) K_{v,h_j}\big((n_{j+1},\ldots,n_M),(y_{j+1},\ldots,y_M)\big).
	\] 
	Therefore, the case $j=1$ is just \eqref{eq:Xtuev}. Moreover, using \eqref{eq:DeltatildeX}, we have
	\begin{align*}
		\Delta^{j}(\widetilde{X}_{v,t}) &= (\Delta^{j-1}\otimes 1)\Delta(\widetilde{X}_{v,t}) \\
		&= (\Delta^{j-1}\otimes 1) (1 \otimes (\widetilde{X}_{v,t}-[t]_qK^{-2}) + \widetilde{X}_{v,t}\otimes K^{-2}) \\
		&= \underbrace{ 1 \otimes \cdots \otimes 1}_{j\ \text{times}}\otimes (\widetilde{X}_{v,t}-[t]_qK^{-2}) + \Delta^{j-1}(\widetilde{X}_{v,t})\otimes K^{-2}.
	\end{align*} 
	Also note that $\pi_N(1)=I$, where $I$ is the identity operator on $H_N$. Now, let $h_j=h_j(\vec{y},t)$ and assume \eqref{eq:deltaXtuev} holds for $j$, then 
	\begin{align*}
		\pi_{\vec{N}}\Big(\Delta^j_\mathrm{L}(\widetilde{X}_{v,t})\Big)K_{v,t}(\cdot,\vec{y}) &= \Big(\pi^{}_{N_{j+1}}(\widetilde{X}_{v,t}-[t]_qK^{-2}) + [h_j]_q  \pi^{}_{N_{j+1}}(K^{-2})\Big)K_{v,t}(\cdot,\vec{y}) \\
		&= \pi_{N_{j+1}}^{}\big(\widetilde{X}_{v,h_{j}} \big)K_{v,t}(\cdot,\vec{y}),
	\end{align*} 
	where $\pi^{}_{N_{j+1}}$ acts on the variable $n_{j+1}$. Since \eqref{eq:Xtuev} holds equally well for $t=h_{j}$, we have
	\[
		\pi_{N_{j+1}}^{}\big(\widetilde{X}_{v,h_{j}} \big)K_{v,t}(\cdot,\vec{y}) = [2x_{j+1}-N_{j+1}+h_j]_q K_{v,t}(\cdot,\vec{y}),
	\]
	thus \eqref{eq:deltaXtuev} is true for $j+1$ as well.
\end{proof}
\begin{remark}
	The proof above shows why the height function $h_j$ appears. Each step of the induction, we have to use \eqref{eq:Xtuev} with $t$ replaced by $h_j$, which adds `$2x_{j+1}-N_{j+1}$' to $h_j$ to obtain the eigenvalue $[2x_{j+1}-N_{j+1}+h_j]_q=[h_{j+1}]_q$. 
\end{remark}
Let us now define the overlap coefficients of the solutions to the GEVPs and EVPs,
\begin{align*}
	R_\mathrm{r}(\vec{x},\vec{y})=R_\mathrm{r}(\vec{x},\vec{y};s,t,v,\vec{N};q)= \langle K_{1,s}(\cdot,\vec{x}),K_{v,t}(\cdot,\vec{y})  \rangle_{H_{\vec{N}}}.
\end{align*}
Since
\[
\langle K_{1,s}(\cdot,\vec{x}),K_{v,t}(	\cdot,\vec{y})  \rangle_{H_{\vec{N}}} = \prod_{j=1}^M \langle k_{1,h^{}_{j-1}(\vec{x},s)}(\cdot,x_j), k_{v,h^{}_{j-1}(\vec{y},t)}(\cdot,y_j) \rangle_{H_{N_j}}^{},
\]
it follows from Proposition \ref{prop:overlaprational} that $R_\mathrm{r}(\vec{x},\vec{y})$ is a nested product of the univariate rational functions $R_\mathrm{r}(x,y)$.
\begin{corollary} We have
	\begin{align*}
		R_\mathrm{r}(\vec{x},\vec{y})=\prod_{j=1}^M R_\mathrm{r}(x_j,y_j;h_{j-1}(s,\vec{x}),h_{j-1}(t,\vec{y}),v,\vec{N};q).
	\end{align*}
\end{corollary}
\begin{remark}
	The nested structure of $R_\mathrm{r}$ is similar to the one of the multivariate $q$-Racah polynomials of Tratnik-type \cite{GRmulti}. 
\end{remark}
The multivariate rational functions $R_\mathrm{r}(\vec{x},\vec{y})$ satisfy a (bi)orthogonality relation. This follows from the orthogonality relations \eqref{eq:KrawOrthn} and \eqref{eq:KrawOrthx} of the univariate $q^{-1}$-Krawtchouk polynomials. Indeed, from those we also get orthogonality relations of the multivariate $q^{-1}$-Krawtchouk polynomials,
\begin{align}
	&\sum_{n^{}_M=0}^{N_M} \cdots \sum_{n^{}_1=0}^{N_1} K_{0,s}(\vec{n},\vec{x})K_{0,s}(\vec{n},\vec{x}')w(\vec{n}) = \frac{\delta_{\vec{x},\vec{x}'}}{W(\vec{x},s;q^{-1})},\label{eq:KrawOrthnMulti}\\
	&\sum_{x^{}_1=0}^{N_1} \cdots \sum_{x^{}_M=0}^{N_M} K_{0,s}(\vec{n},\vec{x})K_{0,s}(\vec{n},\vec{x}')W(\vec{x},s;q^{-1}) = \frac{\delta_{\vec{n},\vec{n}'}}{w(\vec{n})},\label{eq:KrawOrthxMulti}
\end{align}
where the weight function $W(\vec{x},s;q^{-1})$ is a nested product of the univariate ones,
\begin{align*}
	W(\vec{x},s;q^{-1}) &= \prod_{j=1}^M W(x_j,h_{j-1}(\vec{x},s);q^{-1}).
\end{align*}
The first orthogonality relation follows from orthogonality of the univariate $q^{-1}$-Krawtchouk polynomials by first summing over $n^{}_1$, then $n^{}_2$, etc., the second by starting summing over $x^{}_M$, then $x^{}_{M-1}$, up to $x^{}_1$. Since $R_\mathrm{r}(\vec{x},\vec{y})$ is the inner product of (almost) orthogonal basis functions, it is biorthogonal.
\begin{proposition}
	The multivariate rational functions $R_\mathrm{r}(\vec{x},\vec{y})$ satisfy the following biorthogonality relations,
	\begin{align}
		&\sum_{x^{}_1=0}^{N_1}\cdots \sum_{x^{}_M=0}^{N_M} R_\mathrm{r}(\vec{x},\vec{y};s,t,v,\vec{N})\overline{R_\mathrm{r}(\vec{x},\vec{y}{\scriptstyle '};s,t,-\overline{v}-2,\vec{N})}W(\vec{x},s;q^{-1})=\frac{\delta_{\vec{y},\vec{y}\scalebox{.6}{$'$}}}{W(\vec{y},t;q^{-1})}, \label{eq:RatRacMultiOrthx}\\
		&\sum_{y^{}_1=0}^{N_1}\cdots \sum_{y^{}_M=0}^{N_M} R_\mathrm{r}(\vec{x},\vec{y};s,t,v,\vec{N})\overline{R_\mathrm{r}(\vec{x}{\scriptstyle '},\vec{y};s,t,-\overline{v}-2,\vec{N})}W(\vec{y},t;q^{-1})=\frac{\delta_{\vec{x},\vec{x}\scalebox{.6}{$'$}}}{W(\vec{x},s;q^{-1})}.\label{eq:RatRacMultiOrthy}
	\end{align}
\end{proposition}
\begin{proof}
	The proof is similar to the proof of Proposition \ref{prop:biorthraqra} - the univariate case - using the orthogonality of the multivariate $q^{-1}$-Krawtchouk polynomials \eqref{eq:KrawOrthnMulti} and \eqref{eq:KrawOrthxMulti} instead of the univariate ones \eqref{eq:KrawOrthn} and \eqref{eq:KrawOrthx}. 
\end{proof}
\begin{remark}	
	The orthogonality relation of the multivariate rational function $R_\mathrm{r}(\vec{x},\vec{y})$ doesn't trivially follow from the (bi)orthogonality of the univariate rational function $R_\mathrm{r}(x,y)$. This is in contrast with the case of the multivariate $q$-Racah polynomials of Tratnik-type, whose orthogonality relations follow directly by repeated application of the orthogonality relations of the univariate $q$-Racah polynomials. 
\end{remark}
We end this section by deducing $q$-difference equations for the multivariate rational function $R_\mathrm{r}(\vec{x},\vec{y})$. If we would approach this similar as before, we run into a problem. For example, taking $j=1$ and using the GEVP \eqref{eq:GEVPUqmulti} and that both $\Delta^{j-1}_\mathrm{L}(K^2)$ and $\Delta^{j-1}_\mathrm{L}\big(X_{0,s}\big)$ are self-adjoint,
\begin{align}
	\begin{split}[h_{M}(\vec{x},s)]_q &\Big\langle K_{1,s}(\cdot,\vec{x}),\pi_{\vec{N}}\Big( K^2\otimes 1\otimes \cdots \otimes 1\Big)K_{v,t}(\cdot,\vec{y}) \Big\rangle_{H_{\vec{N}}} 	\\
		=  &\Big\langle [h_{M}(\vec{x},s)]_q\pi_{\vec{N}}\Big(K^2\otimes 1\otimes \cdots \otimes 1\Big)K_{1,s}(\cdot,\vec{x}),K_{v,t}(\cdot,\vec{y}) \Big\rangle_{H_{\vec{N}}}\\
		= &\Big\langle \pi_{\vec{N}}\Big( X_{0,s}\otimes 1\otimes \cdots \otimes 1\Big)K_{1,s}(\cdot,\vec{x}),K_{v,t}(\cdot,\vec{y}) \Big\rangle_{H_{\vec{N}}}\\
		= &\Big\langle K_{1,s}(\cdot,\vec{x}),\pi_{\vec{N}}\Big( X_{0,s}\otimes 1\otimes \cdots \otimes 1\Big)K_{v,t}(\cdot,\vec{y}) \Big\rangle_{H_{\vec{N}}}.\end{split}\label{eq:CorGEVPmultiWrong} 
\end{align}
However, if we let $K^2\otimes 1 \otimes \cdots \otimes 1$ act on $K_{v,t}$ as a $q$-difference equation in $y_1$, we lose the desired nested structure of $K_{v,t}$ since $y_1$ appears in $h_{j-1}(\vec{y},t)$, which is a parameter of the other $q^{-1}$-Krawtchouk polynomials appearing as factors in $K_{v,t}$.\\
\\
This can be solved by observing that $K_{1,s}$ is also an eigenfunction of $1\otimes \cdots \otimes 1 \otimes \widetilde{X}_{1,h_{M-1}(\vec{x},s)}$. Indeed, by \eqref{eq:widetildeXEV} with $s$ replaced by $h_{M-1}(\vec{x},s)$ we get
\[
\pi_{\vec{N}}\Big( 1 \otimes \cdots \otimes \widetilde{X}_{1,h_{M-1}(\vec{x},s)}\Big)K_{1,s}(\cdot,\vec{x}) = [h_M(\vec{x},s)]_q K_{1,s}(\cdot,\vec{x}).
\]
More generally, we can use Proposition \ref{prop:multivarkrawEV} to show that
\begin{align}
	\pi_{\vec{N}}\Big(\Delta^{j-1}_\mathrm{R}\big(\widetilde{X}_{1,h_{M-j}(\vec{x},s)} \big)\Big)	K_{1,s}(\cdot,\vec{x}) = [h_M(\vec{x},s)]_q K_{1,s}(\cdot,\vec{x}). \label{eq:deltaXtuev2}
\end{align}
Therefore, $K_{1,s}(\vec{n},\vec{x})$ also solves the GEVP
\begin{align}
	\pi_{\vec{N}}\Big(\Delta^{j-1}_\mathrm{R}\big(X_{0,h_{M-j}(\vec{x},s)} \big)\Big)	K_{1,s}(\cdot,\vec{x}) = \lambda_j \pi_{\vec{N}}\Big(\Delta^{j-1}_\mathrm{R}\big(K^2\big) \Big) K_{1,s}(\cdot,\vec{x}) \label{eq:GEVPUqmulti2}
\end{align}
for $j=1,\ldots,M$ with $\lambda_j=[h_M(\vec{x},s)]_q$. If we now use the GEVP \eqref{eq:GEVPUqmulti2} and that both $\Delta^{j-1}_\mathrm{R}(K^2)$ and $\Delta^{j-1}_\mathrm{R}\big(X_{0,h_{M-j}(\vec{x},s)}\big)$ are self-adjoint, we obtain
\begin{align}
	\begin{split}[h_{M}(\vec{x},s)]_q &\Big\langle K_{1,s}(\cdot,\vec{x}),\pi_{\vec{N}}\Big( \Delta^{j-1}_\mathrm{R}\big(K^2\big)\Big)K_{v,t}(\cdot,\vec{y}) \Big\rangle_{H_{\vec{N}}} 	\\
		=  &\Big\langle [h_{M}(\vec{x},s)]_q\pi_{\vec{N}}\Big( \Delta^{j-1}_\mathrm{R}\big(K^2\big)\Big)K_{1,s}(\cdot,\vec{x}),K_{v,t}(\cdot,\vec{y}) \Big\rangle_{H_{\vec{N}}}\\
		= &\Big\langle \pi_{\vec{N}}\Big( \Delta^{j-1}_\mathrm{R}\big(X_{0,h_{M-j}(\vec{x},s)}\big)\Big)K_{1,s}(\cdot,\vec{x}),K_{v,t}(\cdot,\vec{y}) \Big\rangle_{H_{\vec{N}}}\\
		= &\Big\langle K_{1,s}(\cdot,\vec{x}),\pi_{\vec{N}}\Big( \Delta^{j-1}_\mathrm{R}\big(X_{0,h_{M-j}(\vec{x},s)}\big)\Big)K_{v,t}(\cdot,\vec{y}) \Big\rangle_{H_{\vec{N}}}.\end{split}\label{eq:CorGEVPmulti} 
\end{align}
We can again transfer the action of $\pi_{\vec{N}}\big( \Delta^{j-1}_\mathrm{R}\big(K^2\big)\big)$ and $\pi_{\vec{N}}\big(\Delta^{j-1}_\mathrm{R}\big(X_{0,h_{M-j}(\vec{x},s)}\big)\big)$ from the $\vec{n}$-variable to the $\vec{y}$-variable. Since $\Delta^{j-1}_\mathrm{R}$ is a homomorphism, we know from Lemma \ref{lem:XusinK2Xvt} that we can rewrite $\Delta^{j-1}_\mathrm{R}(X_{0,h_{M-j}(\vec{x},s)})$ in terms of $\Delta^{j-1}_\mathrm{R}(K^2)$ and $\Delta^{j-1}_\mathrm{R}(\widetilde{X}_{v,t})$ for any $t$. Therefore, it would be enough to transfer the $\vec{n}$-dependent action of $\pi_{\vec{N}}\big(\Delta^{j-1}_\mathrm{R}(K^2)\Big)$ and $\pi_{\vec{N}}\big(\Delta^{j-1}_\mathrm{R}(\widetilde{X}_{v,h_{M-j}(\vec{y},t)})\big)$  to the $\vec{y}$-variable. The latter is just the eigenvalue equation \eqref{eq:deltaXtuev2}, the first is a bit more involved. \\

\noindent \textbf{Example.} Let us take $M=4$ and $j=3$. Since
\[
	\Delta^2_\mathrm{R}(K^2) = 1\otimes K^2\otimes K^2\otimes K^2,
\]
we have
\begin{align*}
	\Big[\pi_{\vec{N}}&\Big(\Delta^2_\mathrm{R} (K^2)\Big)K_{v,t}(\cdot,\cdot,\cdot,\cdot,y_1,y_2,y_3,y_4)\Big](n_1,n_2,n_3,n_4)\\
	&=k_{v,t}(n_1,y_1)q^{2n_2-N_2} k_{v,h_1(\vec{y},t)}(n_2,y_2)q^{2n_3-N_3}k_{v,h_2(\vec{y},t)}(n_3,y_3)q^{2n_4-N_4}k_{v,h_3(\vec{y},t)}(n_4,y_4).
\end{align*}
If we apply the regular $q$-difference equation \eqref{eq:3termK2} of the $q$-Krawtchouk polynomials to\\ $k_{v,h_1(\vec{y},t)}(n_2,y_2)$, we get
\begin{align*}
	q^{2n_2-N_2}k_{v,h_1(\vec{y},t)}(n_2,y_2)=& a_{-1}(y,h_1(\vec{y},t)) k_{v,h_1(\vec{y},t)}(n_2,y_2-1)+ a_0(y,h_1(\vec{y},t))k_{v,h_1(\vec{y},t)}(n_2,y_2) \\
	&+ a_1(y,h_1(\vec{y},t))k_{v,h_1(\vec{y},t)}(n_2,y_2+1).
\end{align*}
Therefore, one of the terms we obtain is
\[
k_{v,t}(n_1,y_1)a_{-1}(y,h_1(\vec{y},t)) k_{v,h_1(\vec{y},t)}(n_2,y_2-1)q^{2n_3-N_3}k_{v,h_2(\vec{y},t)}(n_3,y_3)q^{2n_4-N_4}k_{v,h_3(\vec{y},t)}(n_4,y_4).
\]
Applying the regular $q$-difference equation \eqref{eq:3termK2} again on $k_{v,h_2(\vec{y},t)}(n_3,y_3)$, would give a term
\begin{align*}
k_{v,t}(n_1,y_1)a_{-1}(y,h_1(\vec{y},t))k_{v,h_1(\vec{y},t)}(n_2,y_2-1)a_{-1}(y,h_2(\vec{y},t))k_{v,h_2(\vec{y},t)}(n_3,y_3-1)&\\
\times q^{2n_4-N_4}k_{v,h_3(\vec{y},t)}(n_4,y_4)&.
\end{align*}
However, since $k_{v,h_2(\vec{y},t)}(n_3,y_3)$ depends on $y_2$ via the height function, the term above is not of the form required for $K_{v,t}(n_1,n_2,n_3,n_4,y_1,y_2-1,y_3-1,y_4)$. We can solve this by using `dynamical' $q$-difference equations for the $q^{-1}$-Krawtchouk polynomials, which also shift the parameter $t$. The proof can be found in \cite[Lemma 6.1]{GroeneveltWagenaarDyn}, where one has to replace $q$ by $q^{-1}$.
\begin{lemma}\label{lem:qdifkrawtchouk}
	The $q^{-1}$-Krawtchouk polynomials satisfy the following $q$-difference equations,
	\begin{align}
		\begin{split}q^{2n-N}k_{v,t}(n,y) =&  a_{-2,2} k_{v,t+2}(n,y-2)+a_{-1,2} k_{v,t+2}(n,y-1) +a_{0,2} k_{v,t+2}(n,y) ,\end{split}\label{eq:dualqkrawtchoukxrho+}\\
		\begin{split} q^{2n-N}k_{v,t}(n,y) =&  a_{0,-2}k_{v,t-2}(n,y)+a_{1,-2} k_{v,t-2}(n,y+1)+a_{2,-2} k_{v,t-2}(n,y+2),\end{split}\label{eq:dualqkrawtchoukxrho-}
	\end{align}
	where
	\begin{align*}
		a_{-2,2}(y,t)&=\frac{(1-q^{2y})(1-q^{2y-2})}{(1+q^{4y+2t-2N})(1+q^{4y+2t-2N-2})},\\
		a_{-1,2}(y,t)&=\frac{(1+q^{2})(1-q^{2y})(1+q^{2N-2y-2t})}{(1+q^{4y+2t-2N+2})(1+q^{2N-4y-2t+2})}, \\
		a_{0,2}(y,t)&=\frac{(1+q^{2N-2y-2t})(1+q^{2N-2y-2t-2})}{(1+q^{2N-4y-2t})(1+q^{2N-4y-2t-2})},
	\end{align*}
	and
	\begin{align*}
		a_{0,-2}(y,t)&=\frac{(1+q^{2y+2t})(1+q^{2y+2t-2})}{(1+q^{4y+2t-2N})(1+q^{4y+2t-2N-2})},\\
		a_{1,-2}(y,t)&= \frac{(1+q^{2})(1-q^{2N-2y})(1+q^{2y+2t})}{(1+q^{2N-4y-2t+2})(1+q^{4y+2t-2N+2})},\\
		a_{2,-2}(y,t)&= \frac{(1-q^{2N-2y})(1-q^{2N-2y-2})}{(1+q^{2N-4y-2t})(1+q^{2N-4y-2t-2})}.
	\end{align*}
\end{lemma}
\ \\
\noindent \textbf{Example continued.} Therefore, instead of applying the regular $q$-difference equation for\\ $k_{v,h_1(\vec{y},t)}(n_2,y_2)$, we can use \eqref{eq:dualqkrawtchoukxrho-} to obtain a term
\begin{align*}
k_{v,t}(n_1,y_1)a_{-1}(y,h_1(\vec{y},t))k_{v,h_1(\vec{y},t)}(n_2,y_2-1)a_{2,-2}(y,h_2(\vec{y},t))k_{v,h_2(\vec{y}-1,t)}(n_3,y_3+2)&\\
\times q^{2n_4-N_4}k_{v,h_3(\vec{y},t)}(n_4,y_4)&.
\end{align*}
Applying the other `dynamical' equation \eqref{eq:dualqkrawtchoukxrho+} to
\[
q^{2n_4-N_4}k_{v,h_3(\vec{y},t)}(n_4,y_4),
\]
then gives the correct terms. For example, one of these terms would be
\begin{align*}
	k_{v,t}(n_1,y_1)a_{-1}(y,h_1(\vec{y},t))k_{v,h_1(\vec{y},t)}(n_2,y_2-1) a_{2,-2}(y,h_2(\vec{y},t))k_{v,h_2(\vec{y}-1,t)}(n_3,y_3+2)&\\
	\times a_{-1,2}(y,h_3(\vec{y},t))k_{v,h_3(\vec{y}+1,t)}(n_4,y_4-1)&,
\end{align*}
which is equal to
\begin{align*}
	a_{-1}(y,h_1(\vec{y},t))a_{2,-2}(y,h_2(\vec{y},t)a_{-1,2}(y,h_3(\vec{y},t))K_{v,t}(n_1,n_2,n_3,n_4,y_1,y_2-1,y_3+2,y_4-1)&.
\end{align*}
\\
Let us extend this idea to general $j$ and $M$. We can only shift the parameter `$t$' by $\pm 2$. Also, when the parameter `$t$' is shifted to `$t+2$', only terms with `$y$', `$y-1$' and `$y-2$' appear. Similarly, when `$t$' is shifted to `$t-2$', only terms with `$y$', `$y+1$' and `$y+2$' appear . Therefore, the height function $h_j(\vec{y},t)$ may (and can) only stay the same or change $\pm 2$ and the vector $\vec{y}$ can only change to $\vec{y}+\vec{\varepsilon}$, where $\vec{\varepsilon}\in \{0,\pm1,\pm2\}^M$ is such that
\begin{align}
	\begin{split}&\varepsilon_i = 0, \hspace{2.55cm} \text{for }i=1,\ldots,M-j,\\
		&\sum_{i=1}^{\ell} \varepsilon^{}_{i} \in \{-1,0,1\}, \qquad \text{for }\ell=1,\ldots, M. \end{split}\label{eq:conditionepsilon}
\end{align}
Denote the set of all $\vec{\varepsilon}$ that satisfy \eqref{eq:conditionepsilon} by $\mathcal{E}_j$. Note that $\# \mathcal{E}_j = 3^j$. In the example with $M=4$ and $j=3$, there are $27$ terms where the nine terms that have the form $(0,-1,\varepsilon_3,\varepsilon_4)$ are given by
\begin{alignat*}{3}
	&(0,-1,0,0),\qquad &&(0,-1,0,1) ,\qquad &&(0,-1,0,2),\\
	&(0,-1,1,-1),\qquad &&(0,-1,1,0) ,\qquad &&(0,-1,1,1),\\
	&(0,-1,2,-2),\qquad &&(0,-1,2,-1) ,\qquad &&(0,-1,2,0).
\end{alignat*} 
We now have the following multivariate analogue of Lemma \ref{lem:threetermKra}.
\begin{lemma} \label{lem:threetermKraMulti}
	De actions of $\pi_{\vec{N}}\big(\Delta^{j-1}_\mathrm{R}\big(K^2\big)\big)$ and $\pi_{\vec{N}}\big(\Delta^{j-1}_\mathrm{R}\big(X_{0,h_{M-j}(\vec{x},s)}\big)\big)$ on the $\vec{n}$-variable of $K_{v,t}(\vec{n},\vec{y})$ can be transferred to the $\vec{y}$-variable,
	\begin{align}
		\big[\pi_{\vec{N}} \Big(\Delta^{j-1}_\mathrm{R}\big(K^2\big)\Big)K_{v,t}(\cdot,\vec{y})\big](\vec{n}) &= \sum_{\vec{\varepsilon}\in \mathcal{E}_j}A^{(j)}_{\vec{\varepsilon}}(\vec{y},t)K_{v,t}(\vec{n},\vec{y}+\vec{\varepsilon}),\label{eq:K2jntoy}\\
		\big[\pi_{\vec{N}}\Big( \Delta^{j-1}_\mathrm{R}\big(X_{0,h_{M-j}(\vec{x},s)}\big)\Big)K_{v,t}(\cdot,\vec{y})\big](\vec{n}) &= [h_{M-j}(\vec{x},s)]_q K_{v,t}(\vec{n},\vec{y})+ \sum_{\vec{\varepsilon}\in \mathcal{E}_j}B^{(j)}_{\vec{\varepsilon}}(\vec{y},t)K_{v,t}(\vec{n},\vec{y}+\vec{\varepsilon}),\label{eq:Xs0jntoy}
	\end{align}
	where
	\begin{align*}
		A^{(j)}_{\vec{\varepsilon}}(\vec{y},t) &= \prod_{i=M-j+1}^{M} a^{}_{\varepsilon_i,h_{i-1}(\vec{\varepsilon},0)}(y_i,h_{i-1}(\vec{y},t)),\\
		a^{}_{\varepsilon,0}(y,t) &= a^{}_{\varepsilon}(y,t),
	\end{align*} 
	and
	\begin{align*}
		B^{(j)}_{\vec{\varepsilon}}(\vec{y},t) = \begin{cases}
			A^{(j)}_{\vec{\varepsilon}}(\vec{y},t)\big[h_M(\vec{y},t)+v-1\big]_q \qquad &\text{if $\ \sum_{i=1}^M\vec{\varepsilon}_i=-1$,} \\[6pt]
			A^{(j)}_{\vec{\varepsilon}}(\vec{y},t)\big[h_M(\vec{y},t)\big]_q\{v\}_q-[h_{M-j}(\vec{y},t)]_q\{v\}_q \qquad &\text{if $\ \vec{\varepsilon}=0$,} \\[6pt]
			A^{(j)}_{\vec{\varepsilon}}(\vec{y},t)\big[h_M(\vec{y},t)\big]_q\{v\}_q \qquad &\text{if $\ \sum_{i=1}^M\vec{\varepsilon}_i=0$, $\vec{\varepsilon}\neq 0$,}\\[6pt]
			A^{(j)}_{\vec{\varepsilon}}(\vec{y},t)\big[h_M(\vec{y},t)-v+1\big]_q \qquad &\text{if $\ \sum_{i=1}^M\vec{\varepsilon}_i=1$.}
		\end{cases}
	\end{align*}
\end{lemma}
\begin{proof}
	We prove this using induction on $j$. Note that 
	\begin{align*}
		\big[\pi_{\vec{N}}\Big(\Delta^{j-1}_\mathrm{R}\big(K^2\big)\Big)K_{v,t}(\cdot,\vec{y})\big](\vec{n}) = \prod_{i=1}^{M-j} k_{v,h_{i-1}(\vec{y},t)}(n_i,y_i) \prod_{i=M-j+1}^M q^{2n_i-N_i}k_{v,h_{i-1}(\vec{y},t)}(n_i,y_i).
	\end{align*}
	Therefore, the case $j=1$ is just \eqref{eq:3termK2} from Lemma \ref{lem:threetermKra} with `$t$' replaced by `$h_{M-1}(\vec{y},t)$'. If it is true for some $j$, then
	\begin{align*}
		\big[\pi_{\vec{N}}\Big(\Delta^{j}_\mathrm{R}\big(K^2\big)\Big)K_{v,t}(\cdot,\vec{y})\big](\vec{n}) &= \prod_{i=1}^{M-j-1} k_{v,h_{i-1}(\vec{y},t)}(n_i,y_i) \prod_{i=M-j}^M q^{2n_i-N_i}k_{v,h_{i-1}(\vec{y},t)}(n_i,y_i)\\
		&=\sum_{\substack{\vec{\varepsilon}\in \mathcal{E}_{j+1}\\ \vec{\varepsilon}^{}_M=0}}q^{2n^{}_{M}-N^{}_{M}}k_{v,h_{M-1}(\vec{y},t)}(n_M,y_M) \\
		& \qquad \qquad \times C_{\vec{\varepsilon}}\prod_{i=1}^{M-1} k_{v,h_{i-1}(\vec{y}+\vec{\varepsilon},t)}(n_{i},y_{i}+\varepsilon_i),
	\end{align*}
	where 
	\[
	C_{\vec{\varepsilon}} = \prod_{i=M-j}^{M-1} a^{}_{\varepsilon_i,h_{i-1}(\vec{\varepsilon},0)}(y_i,h_{i-1}(\vec{y},t)).
	\]
	Depending on the value of 
	\[
	S=\sum_{i=1}^{M-1} \varepsilon_i,
	\]
	we use a different $q$-difference equation for 
	\[
	q^{2n_{M}-N_{M}}k_{v,h_{M-1}(\vec{y},t)}.
	\]
	If $S=1$ or $S=-1$, we use the `dynamical' $q$-difference equation \eqref{eq:dualqkrawtchoukxrho+} or \eqref{eq:dualqkrawtchoukxrho-} respectively. If $S=0$, we use the regular one \eqref{eq:3termK2}. Therefore, \eqref{eq:K2jntoy} follows.\\
	
	For \eqref{eq:Xs0jntoy}, the proof is similar to the proof of Lemma \ref{lem:threetermKra}. That is, use \eqref{eq:Xs1K2Xtu}, but now with $s$ and $t$ replaced by $h_{M-j}(\vec{x},s)$ and $h_{M-j}(\vec{y},t)$ respectively, and apply $\Delta^{j-1}_\mathrm{R}$ on both sides. Then use \eqref{eq:K2jntoy} for $\pi_{\vec{N}}\big(\Delta^{j-1}_\mathrm{R}(K^2)\big)$ and \eqref{eq:deltaXtuev} for $\pi_{\vec{N}}\big(\Delta^{j-1}_\mathrm{R}\big(\widetilde{X}_{v,h_{M-j}(\vec{y},t)}\big)\big)$, i.e.
	\[
	\Big[\pi_{\vec{N}}\Big(\Delta^{j-1}_\mathrm{R}\big(\widetilde{X}_{v,h_{M-j}(\vec{y},t)}\big)\Big)K_{v,t}(\cdot,\vec{y}+\vec{\varepsilon})\Big](\vec{n}) = h_M(\vec{y}+\vec{\varepsilon},t) K_{v,t}(\vec{n},\vec{y}+\vec{\varepsilon}). \qedhere
	\]
\end{proof}
Now we can use \eqref{eq:CorGEVPmulti} and Lemma \ref{lem:threetermKraMulti} to deduce GEVP's for the multivariate rational functions $R_\mathrm{r}(\vec{x},\vec{y})$, which is a multivariate analogue of Corollary \ref{cor:GEVPR(x,y)}
\begin{corollary} \label{cor:GEVPR(x,y)multi}
	For $j=1,\ldots,M$, the multivariate rational function $R_\mathrm{r}(\vec{x},\vec{y})$ satisfies the following GEVPs,
	\begin{align*}
		[h_M(\vec{x},s)]_q \sum_{\vec{\varepsilon}\in \mathcal{E}_j} A_{\vec{\varepsilon}}^{(j)}(\vec{y},t) R_\mathrm{r}(\vec{x},\vec{y}+\vec{\varepsilon})= [h_{M-j}(\vec{x},s)]_q R_\mathrm{r}(\vec{x},\vec{y}) +\sum_{\vec{\varepsilon}\in \mathcal{E}_j} B_{\vec{\varepsilon}}^{(j)}(\vec{y},t) R_\mathrm{r}(\vec{x},\vec{y}+\vec{\varepsilon}).
	\end{align*}
\end{corollary}
\begin{remark}
	Note that the term
	\[
	[h_{M-j}(\vec{x},s)]_q R_\mathrm{r}(\vec{x},\vec{y})
	\]
	seems a bit odd. However, since $\vec{\varepsilon}\in\mathcal{E}_j$ only has non-zero entries in the last $j$ positions and $h_{M-j}(\vec{x},s)$ does not depend on the last $j$ variables, this factor can be interpreted as the term $[s]_qR_\mathrm{r}(x,y)$ in the GEVP given in Corollary \ref{cor:GEVPR(x,y)}.
\end{remark}	

\section{The algebra $\U_q(\mathfrak{su}_{1,1})$ and rational $\rphisempty{4}{3}$-functions}\label{sec:su11rat}
	This section will be similar to section \ref{sec:su2rat}, but instead of $\U_q(\mathfrak{su}_2)$ and the $q^{-1}$-Krawtchouk polynomials we use $\U_q(\mathfrak{su}_{1,1})$ and the $q^{-1}$-Al-Salam--Chihara polynomials. Consequently, we will work in an infinite dimensional Hilbert space instead of a finite dimensional one. Throughout this section, we take $0<q<1$ and $s,t>-1$.
	\subsection{The algebra $\U_q(\mathfrak{su}_{1,1})$ and twisted primitive elements}
		The quantum algebra $\U_q(\mathfrak{su}_{1,1})$ is the algebra $\U_q(\mathfrak{sl}_2)$ equipped with the $*$-structure which comes from the non-compact Lie algebra $\mathfrak{su}_{1,1}$. This is the anti-linear involution defined on the generators by
	\[
	K^*=K, \quad E^*=-F, \quad F^* = -E, \quad (K^{-1})^* = K^{-1}.
	\]
	Let us define the following (almost) twisted primitive elements in $\U_q(\mathfrak{su}_{1,1})$, 
	\begin{align*}
		Y_{u,s} =& q^{u+\frac12} EK - q^{-u-\frac12}FK + \{s\}_q,\\
		\widetilde{Y}_{u,s}=&q^{-u-\frac12} EK^{-1} - q^{u+\frac12}FK^{-1} + \{s\}_qK^{-2}.
	\end{align*}
	Note the subtle difference with $X_{u,s}$ and $\widetilde{X}_{u,s}$. This difference is to make sure that $Y_{u,s}$ and $\widetilde{Y}_{u,s}$ are self-adjoint in $\U_q(\mathfrak{su}_{1,1})$ if $u\in i\R$ and $s\in\R$. The coproduct of $\widetilde{Y}_{u,s}$ satisfies
	\begin{align}
		\begin{split}\Delta(\widetilde{Y}_{u,s})& = 1 \otimes (\widetilde{Y}_{u,s}-\{s\}_qK^{-2}) + \widetilde{Y}_{u,s}\otimes K^{-2}\\
			&=1 \otimes \widetilde{Y}_{u,0} + \widetilde{Y}_{u,s}\otimes K^{-2} \end{split}\label{eq:DeltatildeY}
	\end{align} 
	Non-surprisingly, we can also write $Y_{u,s}$ as a polynomial in $K^2$ and $\widetilde{Y}_{v,t}$, similar to Lemma \ref{lem:XusinK2Xvt}.
	\begin{corollary} \label{cor:YusinK2Yvt}
		We have
		\begin{align*}
			Y_{u,s}=\frac{q^{u+v}\big[K^2,\widetilde{Y}_{v,t}\big]_q + q^{-u-v} \big[\widetilde{Y}_{v,t},K^2\big]_q }{q^2-q^{-2}}- \{t\}_q\{u+v\}_q+\{s\}_q.
		\end{align*}
	\end{corollary}
	\begin{proof}
		This follows directly from Lemma \ref{lem:XusinK2Xvt} by replacing $$(q^u,q^v,[s]_q,[t]_q)$$ by $$(iq^u,-iq^v,i\{s\}_q,i\{t\}_q)$$ and multiplying both sides by $-i$.
	\end{proof}
	\subsection{A representation of $\U_q(\mathfrak{su}_{1,1})$ and Al-Salam--Chihara polynomials}
	Analogue to the $q$-Krawtchouk polynomials, we now need the Al-Salam--Chihara polynomials, 
	\[
	Q_n(x;a,b;q)= \rphis{3}{2}{q^{-n}, ax, a/x }{ab, 0 }{q,q}.
	\]
	Let $k>0$, we will look at renormalized Al-Salam--Chihara polynomials in base $q^{-2}$, $x$ replaced by $aq^{-2x}$, where $x\in \Z_{\geq 0}$ is discrete, $a=q^{-s-k}$ and $b=q^{s-k}$,
	\begin{align}
		\phi^{}_{u,s}(n,x)&=\phi_{u,s}(n,x;k,q)= q^{n(s-u+\frac12 k +\frac12)} Q_n(q^{-2x-s-k};q^{-s-k},q^{s-k};q^{-2}) \\
		&=q^{n(s-u+\frac12 k +\frac12)}  \rphis{3}{2}{q^{2n}, q^{2x}, q^{-2x-2s-2k}}{q^{-2k}, 0 }{q^{-2},q^{-2}}.
	\end{align}
	Note that 
	\begin{align}
		\phi_{u,s}(n,x)= q^{-un}\phi_{0,s}. \label{eq:phiusphi0s}
	\end{align}
	If $s>-1$, the $q^{-1}$-Al-Salam--Chihara polynomials have an orthogonality relation in $n$ as well as $x$, see \cite{AI}, 
	\begin{align}
		\sum_{n=0}^\infty \phi_{0,s}(n,x)\phi_{0,s}(n,x')w_k(n) = \frac{\delta_{x,x'}}{W_k(x,s)},\label{eq:ASCOrthn}\\
		\sum_{x=0}^\infty \phi_{0,s}(n,x)\phi_{0,s}(n',x)W_k(x,s;q) = \frac{\delta_{n,n'}}{w_k(n)},\label{eq:ASCOrthx}
	\end{align}
	where the weight functions are given by
	\begin{align*}
		w_k(n;q)&=q^{-n(k-1)}\frac{(q^{2k};q^2)_n}{(q^2;q^2)_n},\\
		W_k(x,s;q)&=\frac{1-q^{4x+2s+2k}}{1-q^{2x+2s+2k}}\frac{(q^{2k};q^2)_x}{(q^2;q^2)_x}\frac{(q^{2x+2s +2};q^2)_\infty}{(q^{2x+2s+2k +2};q^2)_\infty}q^{2x(x+s)}.
	\end{align*}
	Next, for this section, we define $H_{k}$ to be the infinite dimensional Hilbert space of functions $f\colon\Z_{\geq 0}\to\C$ with inner product induced by the measure $w_k$,
	\[
	\langle f,g \rangle_{H_k} = \sum_{n=0}^{\infty} f(n) \overline{g(n)} w_k(n).
	\]
	Let $L(H_{k})$ be the space of (possibly unbounded) linear operators on $H_{k}$ and $\pi_{k}\colon\U_q(\mathfrak{su}_{1,1})\to L(H_{k})$ the $*$-representation defined by
	\begin{equation} \label{eq:representationk}
		\begin{split}
			[\pi_k(K)f](n ) &= q^{n+\frac12 k} f(n), \\
			[\pi_k(E)f](n) &= [n]_q f(n-1), \\
			[\pi_k(F) f](n) & = -[n+k]_q f(n+1),\\
			[\pi_k(K^{-1}) f](n) &= q^{-n - \frac12 k}f(n),
		\end{split}
	\end{equation}
	where we take the space $F_0$ of functions with finite support as dense domain. One can easily verify that this is a $*$-representation, i.e. $$\langle\pi_k(X)f,g\rangle_{H_k}=\langle f,\pi_k(X^*)g\rangle_{H_k}$$ for all $X\in \U_q(\mathfrak{su}_{1,1})$ and $f,g\in F_0$.\\
	
	The $q^{-1}$-Al-Salam--Chihara polynomials $\phi_{u,s}(\cdot,x): \Z_{\geq 0}\to \C$ are eigenfunctions of $\pi_k(\widetilde{Y}_{u,s})$,
	\begin{align}
		[\pi_k(\widetilde{Y}_{u,s})\phi_{u,s}(\cdot,x)](n)=\{2x+k+s\}_q\phi_{u,s}(n,x), \label{eq:widetildeYEV}
	\end{align}
	which follows from matching the explicit action of $\pi_k(\widetilde{Y}_{u,s})$ with the three-term recurrence relation of the Al-Salam--Chihara polynomials \cite[(14.8.4)]{KLS}.
	
	\subsection{Rational $\rphisempty{4}{3}$-functions as overlap coefficients}
	Similar to the previous section, we will look at overlap coefficients of a GEVP and an EVP,
	\begin{align}
		\pi_k^{}(Y_{0,s}^{})f&=\lambda \pi_k(K^2)f, \label{eq:GEVPY}\\
		\pi_k^{}(\widetilde{Y}_{v,t}^{})g&=\mu g. \label{eq:EVPY}
	\end{align}
	Again, we rewrite the GEVP \eqref{eq:GEVPY} by multiplying both sides by $\pi_k(K^{-2})$,
	\begin{align}
		\pi_k^{}(\widetilde{Y}_{1,s}^{})f&=\lambda f. \label{eq:GEVPYrewritten}
	\end{align}
	Therefore, 
	\begin{align*}
		f_x(n)= \phi^{}_{1,s}(n,x) \qquad \text{and} \qquad g_y(n)=\phi_{v,t}(n,y)
	\end{align*}
	solve \eqref{eq:GEVPY} and \eqref{eq:EVPY} with
	\begin{align*}
		\lambda_x = \{2x+k+s\}_q \qquad \text{and}\qquad \mu_y = \{2y+k+t\}_q
	\end{align*}
	respectively. Let 
	\begin{equation} \label{eq:innerprod P}
		P_\mathrm{r}(x,y)=P_\mathrm{r}(x,y;s,t,v,k;q) = \langle \phi_{1,s}(\cdot,x),\phi_{v,t}(\cdot,y) \rangle_{H_k}
	\end{equation}
	be the overlap coefficients of these eigenfunctions. Then, using Lemma \ref{lem:summation3phi2to4phi3} with $q$ replaced by $q^2$ and taking $(a,b,c,d)$ to be $$(q^{2k},q^s,q^t,q^{1-v}),$$ we can show that the sum $P_\mathrm{r}$ converges if $|q^{s+t+1-v}|<1$ and that it is a rational function of $\rphisempty{4}{3}$-type.
	\begin{corollary}\label{cor:overlapratASC} If $\mathrm{Re}(v)<1+s+t$,
			\[
		P_\mathrm{r}(x,y)= 	c_1\rphis{4}{3}{q^{-2x},q^{2x+2s+2k},q^{s+t-v-1},q^{s-t-v+1}}{q^{2s+2},q^{-2y+s-t-v+1},q^{2y+s+t+2k-v+1}}{q^2;q^2},
		\]
		where 
		\[
		c_1 = \frac{(q^{s+t+2k-v+1} ;q^2)_\infty}{(q^{s+t-v+1};q^2)_\infty} \frac{(q^{-2y+s-t-v+1};q^2)_{y}}{(q^{s+t+2k-v+1};q^2)_{y}}\frac{(q^{-2x-2s};q^2)_{x}}{(q^{2k};q^2)_{x}}.
		\]
	\end{corollary}
	\subsection{Biorthogonality and $q$-difference relations for $P_\mathrm{r}$}
		Using the orthogonality relations of the $q^{-1}$-Al-Salam--Chihara polynomials, we can derive (bi)orthogonality relations for $P_\mathrm{r}$.
		\begin{proposition}\label{prop:biorthraAW}
			If $|\mathrm{Re}(v)+1|<2+s+t$, we have
			\begin{align}
				&\sum_{x=0}^\infty P_\mathrm{r}(x,y;s,t,v,k)\overline{P_\mathrm{r}(x,y';s,t,-\overline{v}-2,k)}W_k(x,s;q)=\frac{\delta_{y,y'}}{W_k(y,t;q)}, \label{eq:RatAWOrthx}\\
				&\sum_{y=0}^\infty P_\mathrm{r}(x,y;s,t,v,k)\overline{P_\mathrm{r}(x',y;s,t,-\overline{v}-2,k)}W_k(y,t;q)=\frac{\delta_{x,x'}}{W_k(x,s;q)}.\label{eq:RatAWOrthy}
			\end{align}
		\end{proposition}
		\begin{proof}
			For this proof, let $\mathcal{H}_k^s$ be the Hilbert space with inner product given by
			\begin{align*}
				\langle f,g\rangle_{\mathcal{H}_k^s} = \sum_{x=0}^{\infty}f(x)\overline{g(x)}W_k(x,s).
			\end{align*}
			Let $\Lambda_s\colon H_k\to\mathcal{H}_k^s$ be the linear operator defined by
			\[
			(\Lambda_s f)(n)= \langle f, \phi_{0,s}(\cdot,x)\rangle_{H_k} = \sum_{n=0}^\infty f(n)\phi_{0,s}(n,x)w_k(n).
			\]
			Now, the proof is similar to the one of Proposition \ref{prop:biorthraqra}. It still works in the infinite dimensional case since $\phi_{u,s}(\cdot,y)\in H_k$ is in the linear span of delta-functions. Note that we need the sum in \eqref{eq:innerprod P} for both $P_\mathrm{r}(x,y;s,t,v,k)$ and $P_\mathrm{r}(x,y';s,t,-\overline{v}-2,k)$ to converge. For the first, we require $\text{Re}(v)<s+t+1$ (see Corollary \ref{cor:overlapratASC}). For the second, we replace $v$ by $-\overline{v}-2$ to get $-\text{Re}(v)-2 < s+t+1$. Combining these two constraints gives $|\text{Re}(v)+1|<s+t+2$. \\
			\indent The second biorthogonality relation \eqref{eq:RatAWOrthy} follows directly from the first since we have the symmetry $P_\mathrm{r}(y,x;t,s)=P_\mathrm{r}(x,y;s,t)$.
		\end{proof}
		\begin{remark}
			Similar to Remark \ref{rem:biorth=orth}, we have orthogonality when $\text{Re}(v)=-1$.
		\end{remark}
		Let us now turn to the recurrence relation for $P_\mathrm{r}$. Similar to \eqref{eq:CorGEVP}, we want to use the GEVP \eqref{eq:GEVPY} to show that
		\begin{align}
		\{2x+k+s\}_q \Big\langle \phi_{1,s}(\cdot,x),\pi_k(K^2)\phi_{v,t}(\cdot,y) \Big\rangle_{H_k} = \Big\langle \phi_{1,s}(\cdot,x),\pi_k(Y_{0,s})\phi_{v,t}(\cdot,y) \Big\rangle_{H_k}. \label{eq:GEVPtorecrel}
		\end{align}
		We know that $\pi_k(K^2)$ and $\pi_k(Y_{0,s})$ are symmetric with respect the inner product restricted to finitely supported functions, but not necessarily self-adjoint. Since $\pi_k(K^2)$ is bounded, it is self-adjoint. However, $\pi_k(Y_{0,s})$ is unbounded and we still have to show it is symmetric in our case, i.e.
		\begin{align}
			\Big\langle \pi_k(Y_{0,s})\phi_{1,s}(\cdot,x),\phi_{v,t}(\cdot,y) \Big\rangle_{H_k}= \Big\langle \phi_{1,s}(\cdot,x),\pi_k(Y_{0,s})\phi_{v,t}(\cdot,y) \Big\rangle_{H_k}. \label{eq:Y0ssym}
		\end{align}
		This can be proved using the dominated convergence theorem and the following Lemma, which is the analogue of Lemma \ref{lem:threetermKra}, showing that we can transfer the action of $\pi_k(K^2)$ and $\pi_k(Y_{0,s})$ on the $n$-variable of $\phi_{v,t}$ to the $y$-variable.
		\begin{lemma}\label{lem:threetermASC}
			We have
			\begin{align}
				\left[\pi_k(K^2)\phi_{v,t}(\cdot,y)\right](n) &= c_{-1}^{} \phi_{v,t}(n,y-1) + c_{0}^{}\phi_{v,t}(n,y)+ c_{1}^{} \phi_{v,t}(n,y+1),\label{eq:3termK2phi}\\
				\left[\pi_k(Y_{0,s})\phi_{v,t}(\cdot,y)\right](n) &=  d_{-1}^{} \phi_{v,t}(n,y-1) + (d_{0}^{}+\{s\}_q) \phi_{v,t}(n,y)+ d_{1}^{} \phi_{v,t}(n,y+1),\label{eq:3termY}
			\end{align}
			where
			\begin{align*}
				c_{-1}^{}(y,t) &= -\frac{q^{-4y-2t-3k+2}(1-q^{-2y})(1-q^{-2y-2t})}{(1+q^{-4y-2t-2k+2})(1-q^{-4y-2t-2k})}, \\
				c_{0}^{}(y,t) &= -(c_{-1}(y)+c_{1}(y)-1),\\
				c_{1}^{}(y,t) &= \frac{q^{k}(1-q^{-2y-2k})(1+q^{-2y-2t-2k})}{(1+q^{-4y-2t-2k})(1+q^{-4y-2t-2k-2})},
			\end{align*}
			and
			\begin{align*}
				d_{-1}(y,t)&= c_{-1}(y)\{2y+t+k+v-1\}_q ,\\
				d_{0}(y,t)&= c_0(y)\{2y+k+t\}_q\{v\}_q- \{t\}_q\{v\}_q ,\\
				d_{1}(y,t)&= c_1(y)\{2y+k+t-v+1\}_q .
			\end{align*}
		\end{lemma}
		\begin{proof}
			Since $\pi_N^{}(K^2)$ is just multiplication by $q^{2n+k}$, the first equation \eqref{eq:3termK2} follows from the $q$-difference equation of the Al-Salam--Chihara polynomials (\cite[(14.8.7)]{KLS} after replacing $q$ by $q^{-2}$, $z$ by $aq^{-2x}$, and substituting $(a,b)=(q^{-s-k},q^{s-k})$.\\
			\indent The other equation follows from Corollary \ref{cor:YusinK2Yvt}, similarly as done in the proof of Lemma \ref{lem:threetermKra}.
		\end{proof}
		\noindent Now we are ready to prove a recurrence relation for $P_\mathrm{r}$.
		\begin{proposition}\label{prop:GEVPP(x,y)}
			The rational function $P_\mathrm{r}(x,y)$ satisfies
			\begin{multline*}
				\{2x+k+s\}_q\Big(c_{-1}^{}(y,t) P_\mathrm{r}(x,y-1) + c_{0}^{}(y,t) P_\mathrm{r}(x,y)+ c_{1}^{}(y,t) P_\mathrm{r}(x,y+1)\Big) \\
				= d_{-1}(y,t) P_\mathrm{r}(x,y-1) + (d_{0}(y,t)+\{s\}_q)(y) P_\mathrm{r}(x,y)+ d_{1}(y,t) P_\mathrm{r}(x,y+1).
			\end{multline*}
		\end{proposition}
		\begin{proof}
			This follows from \eqref{eq:GEVPtorecrel} using Lemma \ref{lem:threetermASC} and the definition of $P_\mathrm{r}(x,y)$. However, we still need to prove \eqref{eq:Y0ssym}, before we can use \eqref{eq:GEVPtorecrel}. For each $N\in \Z_{\geq 0}$, define
			\[
				\phi^N_{v,t}(n,y)= \begin{cases}
					\phi_{v,t}(n,y)\qquad &\text{if}\ n\leq N,\\
					0\qquad &\text{else}.
				\end{cases}
			\]
			Then 
			\begin{align}
				\Big\langle \pi_k(Y_{0,s})\phi^N_{1,s}(\cdot,x),\phi_{v,t}^N(\cdot,y) \Big\rangle_{H_k}= \Big\langle \phi^N_{1,s}(\cdot,x),\pi_k(Y_{0,s})\phi^N_{v,t}(\cdot,y) \Big\rangle_{H_k}, \label{eq:Y0ssymM}
			\end{align}
			since $\pi_k(Y_{0,s})$ is symmetric with respect to functions with finite support. Rewriting above equation gives
			\begin{align}
				\sum_{n=0}^\infty \big[\pi_k(Y_{0,s})\phi^N_{1,s}(\cdot,x)\big](n)\phi_{v,t}^N(n,y)w_k(n)- \phi^N_{1,s}(n,x)\big[\pi_k(Y_{0,s})\phi^N_{v,t}(\cdot,y)\big](n) w_k(n) = 0. \label{eq:sumDCT}
			\end{align} 
			We then take the limit $N\to\infty$ and use the dominated convergence theorem to pull this limit into the infinite sum. To justify this, we use the explicit action of the tridiagonal operator $\pi_k(Y_{0,s})$ to get
			\begin{align*}
				\big[\pi_k(Y_{0,s})\phi^N_{v,t}(\cdot,y)\big](n) = \begin{cases}
					\big[\pi_k(Y_{0,s})\phi_{v,t}(\cdot,y)\big](n) &\text{ if } n\leq N-1,\\
					\big[\pi_k(Y_{0,s})\phi_{v,t}(\cdot,y)\big](n) +q^{N+\half k+\half}[N+k]_q \phi_{v,t}(N+1,y) &\text{ if } n=N,\\
					q^{N+\half k+\half}[N+1]_q \phi_{v,t}(N,y) &\text{ if } n=N+1,\\
					0 &\text{ else.}
				\end{cases}
			\end{align*}
			Thus, by Lemma \ref{lem:threetermASC} and the estimate
			\[
			\left|q^{N+a}[N+b]_q\right|  = \left| \frac{q^{2N + a + b}-q^{a-b}}{q-q^{-1}}\right| \leq \frac{q^{a+b}+q^{a-b}}{|q-q^{-1}|},
			\]
			we know that there exist constants $C_1,C_2,C_3,C_4$, independent of $N$, such that 
			\begin{align*}
				\left| \big[\pi_k(Y_{0,s})\phi^N_{v,t}(\cdot,y)\big](n)\right| \leq& C_1 |\phi_{v,t}(n,y-1)| + C_2 |\phi_{v,t}(n,y)| + C_3 |\phi_{v,t}(n,y+1)| \\
				&+ C_4 |\phi_{v,t}(n+1,y)|+ C_5 |\phi_{v,t}(n-1,y)|.
			\end{align*}
			Using these estimate for \eqref{eq:sumDCT}, gives a dominating function. To see that this function is integrable, apply the Cauchy-Schwarz inequality and use that $\phi_{v,t}(\cdot,y)\in H_k$ for each $y\in \Z_{\geq 0}$.
		\end{proof}
	\subsection{Multivariate rational functions as overlap coefficients in $\U_q(\mathfrak{su}_{1,1})^{\otimes M}$}
		Let $\vec{k}=(k_1,\ldots,k_M)$ with each $k_j >0$, then we denote by $H_{\vec{k}}$ the closure (with respect to the inherited norm) of the algebraic tensor product $H_{k_1}\otimes\cdots\otimes H_{k_M}$. This Hilbert space is isomorphic to the Hilbert space of functions on $\Z_{\geq 0}^M$ with  inner product
		\[
			\langle f,g\rangle_{H_{\vec{k}}} = \sum_{n_1=0}^\infty \cdots \sum_{n_M=0}^\infty f(\vec{n})\overline{g(\vec{n})}w_{\vec{k}}(\vec{n}),
		\]
		where $w_{\vec{k}}(\vec{n})$ is the product of the univariate weights $w_{k_j}$,
		\[
			w_{\vec{k}}^{}(\vec{n}) = \prod_{j=1}^{M}w^{}_{k_j}(n_j).
		\]
		Now we can define the following multivariate analogue of the GEVP \eqref{eq:GEVPY} and EVP \eqref{eq:EVPY}. We want to find $f,g \in H_{\vec{k}}$ such that for all $j=1,\ldots,M,$
	\begin{align}
		\pi_{\vec{k}}  \Big(\Delta^{j-1}_\mathrm{L}\big(Y_{0,s}\big)\Big)f&=\lambda_j \pi_{\vec{k}}^{}\Big(\Delta^{j-1}_\mathrm{L}\big(K^2\big)\Big) f, \label{eq:GEVPUqmultiY}\\
		\pi_{\vec{k}}\Big(\Delta^{j-1}_\mathrm{L}\big(\widetilde{Y}_{v,t}\big)\Big)g&= \mu_j g. \label{eq:EVPYtmulti}
	\end{align}
	Again, we can rewrite \eqref{eq:GEVPUqmultiY} to an EVP by multiplying both sides by $\pi_{\vec{k}}^{}\big(\Delta^{j-1}_\mathrm{L}(K^{-2})\big)$,
	\begin{align}
		\pi_{\vec{k}}\Big(\Delta^{j-1}_\mathrm{L}(\widetilde{Y}_{1,s})\Big)f&= \lambda_j f.\label{eq:GEVPrewrittenmultiY}
	\end{align}
	 Solving these EVPs goes similar as for the finite dimensional case. We first show that a nested product of $q^{-1}$-Al-Salam--Chihara polynomials are eigenfunctions of $\pi_{\vec{k}}\big(\Delta^{j-1}_\mathrm{L}(\widetilde{Y}_{v,t})\big)$. Define the height function $h'$ by
	\begin{align*}
		h'_{j}=h_j'(\vec{y},t)=t + \sum_{i=1}^j 2y_i + k_i,
	\end{align*} 
	which can be obtained from the height function $h$ by replacing each $N_i$ by $-k_i$. Let
	\begin{align*}
		\Phi_{v,t}\big(\vec{n},\vec{y}\big) = \prod_{j=1}^{M} \phi^{}_{v,h'_{j-1}(\vec{y},t)}(n_j,y_j),
	\end{align*}
	Then we have the following proposition, which is the analogue of Proposition \ref{prop:multivarkrawEV}, which is proved in the same way.
	\begin{proposition}\label{prop:multivarASCEV}
		$\Phi_{v,t}\big(\vec{n},\vec{y}\big)$	are eigenfunctions of $\pi_{\vec{k}}\big(\Delta^{j-1}_\mathrm{L}(\widetilde{Y}_{v,t})\big)$ for $j=1,...,M$,
		\begin{align}
			\big[\pi_{\vec{k}}\Big(\Delta^{j-1}_\mathrm{L}(\widetilde{Y}_{v,t})\Big)\Phi_{v,t}(\,\cdot\,,\vec{y})\big](\vec{n}) = \{h'_{j}(\vec{y},t)\}_q \Phi_{v,t}(\vec{n},\vec{y}). \label{eq:deltaYtuev}
		\end{align}
	\end{proposition}
	\noindent From this proposition we get that $\Phi_{1,s}(\cdot,\vec{y})$ and $\Phi_{v,t}(\cdot,\vec{x})$ solve \eqref{eq:GEVPUqmultiY} and \eqref{eq:EVPYtmulti} with $$\lambda_j = \{h'_j(\vec{x},s)\}_q\qquad \text{and}\qquad \mu_j = \{h'_j(\vec{y},t)\}_q$$ respectively.
	Let
	\begin{align*}
		P_\mathrm{r}(\vec{x},\vec{y})=P_\mathrm{r}(\vec{x},\vec{y};s,t,v,\vec{k};q)= \langle \Phi_{1,s}(\,\cdot\,,\vec{x}),\Phi_{v,t}(\,\cdot\,,\vec{y})  \rangle_{H_{\vec{k}}}
	\end{align*}
	be the overlap coefficients of the solutions to these GEVPs and EVPs.
	Since
	\[
	\langle \Phi_{1,s}(\,\cdot\,,\vec{x}),\Phi_{v,t}(\,\cdot\,,\vec{y})  \rangle_{H_{\vec{k}}} = \prod_{j=1}^M \langle \phi_{1,h'_{j-1}(\vec{x},s)}(\cdot,x_j), \phi_{v,h'_{j-1}(\vec{y},t)}(\cdot,y_j) \rangle_{H_{k_j}}^{},
	\]
	it follows from applying Corollary \ref{cor:overlapratASC} $M$ times that $P_\mathrm{r}(\vec{x},\vec{y})$ is a nested product of the univariate rational functions $P_\mathrm{r}(x,y)$.
	\begin{corollary} We have
		\begin{align*}
			P_\mathrm{r}(\vec{x},\vec{y})=\prod_{j=1}^M P_\mathrm{r}(x_j,y_j;h'_{j-1}(s,\vec{x}),h_{j-1}'(t,\vec{y}),v,\vec{k}).
		\end{align*}
	\end{corollary}
	Let us next consider orthogonality relations for $P_\mathrm{r}(\vec{x},\vec{y})$. From the 
	orthogonality relations \eqref{eq:ASCOrthn} and \eqref{eq:ASCOrthx} of the univariate $q^{-1}$-Al-Salam--Chihara polynomials we get orthogonality for the multivariate ones. That is, for $s>-1$ we have,
	\begin{align}
		&\sum_{n^{}_M=0}^{\infty} \cdots \sum_{n^{}_1=0}^{\infty} \Phi_{0,s}(\vec{n},\vec{x})\Phi_{0,s}(\vec{n},\vec{x}')w_{\vec{k}}(\vec{n}) = \frac{\delta_{\vec{x},\vec{x}'}}{W_{\vec{k}}(x,s)},\label{eq:ASCOrthnMulti}\\
		&\sum_{x^{}_1=0}^{\infty} \cdots \sum_{x^{}_M=0}^{\infty} \Phi_{0,s}(\vec{n},\vec{x})\Phi_{0,s}(\vec{n},\vec{x}')W_{\vec{k}}(\vec{x},s) = \frac{\delta_{\vec{n},\vec{n}'}}{w_{\vec{k}}(\vec{n})},\label{eq:ASCOrthxMulti}
	\end{align}
	where the weight function $W^{}_{\vec{k}}(\vec{x},s)$ is a (nested) product of the univariate ones,
	\begin{align*}
		W_{\vec{k}}(\vec{x},s) &= \prod_{j=1}^M W_{k_j}(x_j,h'_{j-1}(\vec{x},s)).
	\end{align*}
	Since $P_\mathrm{r}(\vec{x},\vec{y})$ are overlap coefficients of (almost) orthogonal bases, it satisfies (bi)orthogonality relations.
	\begin{proposition} \label{prop:biorthraAWmulti}
			Let $|\text{Re}(v)+1|<2+s+t$, then the multivariate rational functions $P_\mathrm{r}(\vec{x},\vec{y})$ satisfy the following orthogonality relations,
		\begin{align}
			&\sum_{x^{}_1=0}^{\infty}\cdots \sum_{x^{}_M=0}^{\infty} P_\mathrm{r}(\vec{x},\vec{y};s,t,v,\vec{k})\overline{P_\mathrm{r}(\vec{x},\vec{y}\hspace{0.015cm}';s,t,-\overline{v}-2,\vec{k})}W_{\vec{k}}(\vec{x},s)=\frac{\delta_{\vec{y},\vec{y}' }}{W_{\vec{k}}(\vec{y},t)}, \label{eq:RatAWMultiOrthx}\\
			&\sum_{y^{}_1=0}^{\infty}\cdots \sum_{y^{}_M=0}^{\infty} P_\mathrm{r}(\vec{x},\vec{y};s,t,v,\vec{k})\overline{P_\mathrm{r}(\vec{x}\hspace{0.015cm}',\vec{y};s,t,-\overline{v}-2,\vec{k})}W_{\vec{k}}(\vec{y},t)=\frac{\delta_{\vec{x},\vec{x}' }}{W_{\vec{k}}(\vec{x},s)}.\label{eq:RatAWMultiOrthy}
		\end{align}
	\end{proposition}
	\begin{proof}
		Similar to the proof of Proposition \ref{prop:biorthraAW}, using the orthogonality for the multivariate $q^{-1}$-Al-Salam--Chihara polynomials instead of the univariate one.
	\end{proof}
	We end this section by deducing $q$-difference equations for the multivariate rational functions $P_\mathrm{r}(\vec{x},\vec{y})$. We use a similar approach as for $R_\mathrm{r}(x,y)$. Using Proposition \ref{prop:multivarASCEV}, we have that $\Phi_{1,s}(\cdot,\vec{x})$ are also eigenfunctions of $\pi_{\vec{k}}\Big(\Delta^{j-1}_\mathrm{R}\big(\widetilde{Y}_{1,h'_{M-j}(\vec{x},s)} \big)\Big)$,
	\begin{align}
		\pi_{\vec{k}}\Big(\Delta^{j-1}_\mathrm{R}\big(\widetilde{Y}_{1,h'_{M-j}(\vec{x},s)} \big)\Big)	\Phi_{1,s}(\,\cdot\,,\vec{x}) = \{h'_M(\vec{x},s)\}_q \Phi_{1,s}(\,\cdot\,,\vec{x}). \label{eq:deltaYtuev2}
	\end{align}
	Therefore, $\Phi_{1,s}(\vec{n},\vec{x})$ also solves the GEVP
	\begin{align}
		\pi_{\vec{k}}\Big(\Delta^{j-1}_\mathrm{R}\big(Y_{0,h'_{M-j}(\vec{x},s)} \big)\Big)	\Phi_{1,s}(\,\cdot\,,\vec{x}) = \lambda_j \pi_{\vec{k}}\Big(\Delta^{j-1}_\mathrm{R}\big(K^2\big) \Big) \Phi_{1,s}(\,\cdot\,,\vec{x}) \label{eq:GEVPUqmulti2Y}
	\end{align}
	for $j=1,\ldots,M$ with $\lambda_j=\{h_M(\vec{x},s)\}_q$. If we now use the GEVP \eqref{eq:GEVPUqmulti2Y} and that both $\Delta^{j-1}_\mathrm{R}(K^2)$ and $\Delta^{j-1}_\mathrm{R}\big(Y_{0,h'_{M-j}(\vec{x},s)}\big)$ are symmetric, we obtain
	\begin{align}
		\begin{split}\{h'_{M}(\vec{x},s)\}_q &\Big\langle \Phi_{1,s}(\,\cdot\,,\vec{x}),\pi_{\vec{k}}\Big( \Delta^{j-1}_\mathrm{R}\big(K^2\big)\Big)\Phi_{v,t}(\,\cdot\,,\vec{y}) \Big\rangle_{H_{\vec{k}}} \\
			= &\Big\langle \Phi_{1,s}(\,\cdot\,,\vec{x}),\pi_{\vec{k}}\Big( \Delta^{j-1}_\mathrm{R}\big(Y_{0,h'_{M-j}(\vec{x},s)}\big)\Big)\Phi_{v,t}(\,\cdot\,,\vec{y}) \Big\rangle_{H_{\vec{k}}}.\end{split}\label{eq:CorGEVPmultiY} 
	\end{align}
	Now again, 
	\begin{itemize}
		\item we can rewrite $\Delta^{j-1}_\mathrm{R}(Y_{0,h'_{M-j}(\vec{x},s)})$ in terms of $\Delta^{j-1}_\mathrm{R}(K^2)$ and $\Delta^{j-1}_\mathrm{R}(\widetilde{Y}_{v,t})$ for any $t$,
		\item the action of $\pi_{\vec{k}}\big( \Delta^{j-1}_\mathrm{R}\big(K^2\big)\big)$ and $\Delta^{j-1}_\mathrm{R}(\widetilde{Y}_{v,t})$ can be transferred from the $\vec{n}$-variable to the $\vec{y}$-variable.
	\end{itemize}
	The first point is a consequence of Corollary \ref{cor:YusinK2Yvt}, the second one of the eigenvalue equation \eqref{eq:deltaYtuev2} and the following analogue of Lemma \ref{lem:threetermKraMulti}.
	\begin{lemma} \label{lem:threetermASCMulti}
		De actions of $\pi_{\vec{k}}\big(\Delta^{j-1}_\mathrm{R}\big(K^2\big)\big)$ and $\pi_{\vec{k}}\big(\Delta^{j-1}_\mathrm{R}\big(Y_{0,h'_{M-j}(\vec{x},s)}\big)\big)$ on the $\vec{n}$-variable of $\Phi_{v,t}(\vec{n},\vec{y})$ can be transferred to the $\vec{y}$-variable,
		\begin{align*}
			\big[\pi_{\vec{k}} \Big(\Delta^{j-1}_\mathrm{R}\big(K^2\big)\Big)\Phi_{v,t}(\,\cdot\,,\vec{y})\big](\vec{n}) &= \sum_{\vec{\varepsilon}\in \mathcal{E}_j}C^{(j)}_{\vec{\varepsilon}}(\vec{y},t)\Phi_{v,t}(\vec{n},\vec{y}+\vec{\varepsilon}),\\
			\big[\pi_{\vec{N}}\Big( \Delta^{j-1}_\mathrm{R}\big(Y_{0,h'_{M-j}(\vec{x},s)}\big)\Big)\Phi_{v,t}(\,\cdot\,,\vec{y})\big](\vec{n}) &= \{h_{M-j}(\vec{x},s)\}_q \Phi_{v,t}(\vec{n},\vec{y})+ \sum_{\vec{\varepsilon}\in \mathcal{E}_j}D^{(j)}_{\vec{\varepsilon}}(\vec{y},t)\Phi_{v,t}(\vec{n},\vec{y}+\vec{\varepsilon}),
		\end{align*}
		where
		\begin{align*}
			C^{(j)}_{\vec{\varepsilon}}(\vec{y},t) &= \prod_{i=M-j+1}^{M} c^{}_{\varepsilon_i,h'_{i-1}(\vec{\varepsilon},0)}(y_i,h'_{i-1}(\vec{y},t)),\\
			c^{}_{\varepsilon,0}(y,t) &= c^{}_{\varepsilon}(y,t),
		\end{align*} 
		and
		\begin{align*}
			D^{(j)}_{\vec{\varepsilon}}(\vec{y},t) = \begin{cases}
				C^{(j)}_{\vec{\varepsilon}}(\vec{y},t)\{h'_M(\vec{y},t)+v-1\}_q \qquad &\text{if $\ \sum_{i=1}^M\vec{\varepsilon}_i=-1$,} \\[6pt]
				C^{(j)}_{\vec{\varepsilon}}(\vec{y},t)\{h'_M(\vec{y},t)\big\}_q\{v\}_q-\{h'_{M-j}(\vec{y},t)\}_q\{v\}_q \qquad &\text{if $\ \vec{\varepsilon}=0$,} \\[6pt]
				C^{(j)}_{\vec{\varepsilon}}(\vec{y},t)\{h'_M(\vec{y},t)\}_q\{v\}_q \qquad &\text{if $\ \sum_{i=1}^M\vec{\varepsilon}_i=0$, $\vec{\varepsilon}\neq 0$,}\\[6pt]
				C^{(j)}_{\vec{\varepsilon}}(\vec{y},t)\{h'_M(\vec{y},t)-v+1\}_q \qquad &\text{if $\ \sum_{i=1}^M\vec{\varepsilon}_i=1$.}
			\end{cases} .
		\end{align*}
	\end{lemma}
	\begin{proof}
		Similar to the proof of Lemma \ref{lem:threetermKraMulti}, where now one has to use \eqref{eq:3termK2phi} and the following `dynamical' $q$-difference equations for the univariate Al-Salam--Chihara polynomials,
		\begin{align*}
			\begin{split}q^{2n+k}\phi_{v,t}(n,y) =&  c_{-2,2}(y,t) \phi_{v,t+2}(n,y-2)+c_{-1,2}(y,t) \phi_{v,t+2}(n,y-1) +c_{0,2}(y,t) \phi_{v,t+2}(n,y) ,\end{split}\\
			\begin{split} q^{2n+k}\phi_{v,t}(n,y) =&  c_{0,-2}(y,t) \phi_{v,t-2}(n,y)  +c_{1,-2}(y,t) \phi_{v,t-2}(n,y+1)+c_{2,-2}(y,t) \phi_{v,t-2}(n,y+2),\end{split}
		\end{align*}
		where
		\begin{align*}
			c_{-2,2}(y,t)&=\frac{(1-q^{2y})(1-q^{2y-2})}{(1+q^{4y+2t-2N})(1+q^{4y+2t-2N-2})},\\
			c_{-1,2}(y,t)&=\frac{(1+q^{2})(1-q^{2y})(1-q^{-2y-2k-2t})}{(1-q^{4y+2t+2k+2})(1-q^{-4y-2k-2t+2})}, \\
			c_{0,2}(y,t)&=\frac{(1-q^{-2y-2k-2t})(1-q^{-2y-2k-2t-2})}{(1-q^{-4y-2k-2t})(1-q^{-4y-2k-2t-2})},
		\end{align*}
		and
		\begin{align*}
			c_{0,-2}(y,t)&=\frac{(1-q^{2y+2t})(1-q^{2y+2t-2})}{(1-q^{4y+2t+2k})(1-q^{4y+2t+2k-2})},\\
			c_{1,-2}(y,t)&= \frac{(1+q^{2})(1-q^{-2y-2k})(1-q^{2y+2t})}{(1-q^{-4y-2k-2t+2})(1-q^{4y+2t+2k+2})},\\
			c_{2,-2}(y,t)&= \frac{(1-q^{-2y-2k})(1-q^{-2y-2k-2})}{(1-q^{-4y-2k-2t})(1-q^{-4y-2k-2t-2})}.
		\end{align*}
		These equations can be found\footnote{In \cite[Lemma 5.9]{Gr21}, replace $q$ by $q^{-1}$, $s$ by $q^{-t}$ and $x$ by $q^{-2y-t-k}$.} in \cite[Lemma 5.9]{Gr21}. 
	\end{proof}
	We can now combine \eqref{eq:CorGEVPmultiY} and Lemma \ref{lem:threetermASCMulti} to deduce recurrence relations for the multivariate rational functions $P_\mathrm{r}(\vec{x},\vec{y})$, which generalize Proposition \ref{prop:GEVPP(x,y)}.
	\begin{corollary} \label{cor:GEVPP(x,y)multi}
		For $j=1,\ldots,M$, the multivariate rational function $P_\mathrm{r}(\vec{x},\vec{y})$ satisfies the following GEVPs,
		\begin{align*}
			\{h_M(\vec{x},s)\}_q \sum_{\vec{\varepsilon}\in \mathcal{E}_j} C_{\vec{\varepsilon}}^{(j)}(\vec{y},t) P_\mathrm{r}(\vec{x},\vec{y}+\vec{\varepsilon})= \{h_{M-j}(\vec{x},s)\}_q P_\mathrm{r}(\vec{x},\vec{y}) +\sum_{\vec{\varepsilon}\in \mathcal{E}_j} D_{\vec{\varepsilon}}^{(j)}(\vec{y},t) P_\mathrm{r}(\vec{x},\vec{y}+\vec{\varepsilon}).
		\end{align*}
	\end{corollary}
	\section{Application: $R_\mathrm{r}$ and $P_\mathrm{r}$ as duality functions of interacting particle systems} \label{sec:duality}
	As an application of the multivariate rational functions $R_\mathrm{r}(x,y)$ and $P_\mathrm{r}(x,y)$, we will show that they appear as duality functions for certain interacting particle systems. Let $\{\mathcal{X}_t\}_{t\geq0}$ and $\{\widehat{\mathcal{X}}_t\}_{t\geq0}$ be Markov processes with state spaces $\Omega$ and $\widehat{\Omega}$ and generators $L$ and $\widehat{L}$. We say that $\mathcal{X}_t$ and $\widehat{\mathcal{X}}_t$ are dual to each other with respect to a duality function $D:\Omega\times\widehat{\Omega}\to \C$ if
	\[		
	[L D(\,\cdot\,,\hat{\eta})](\eta)=[\widehat{L} D(\eta,\,\cdot\,)](\hat{\eta})
	\]  
	for all $\eta\in \Omega$ and $\hat{\eta}\in\widehat{\Omega}$. It was shown in \cite{GroeneveltWagenaarDyn} and \cite{Wa} that $K_{v,t}(\vec{n},\vec{x})$ and $\Phi_{v,t}(\vec{n},\vec{x})$ are duality functions: the first for the duality between ASEP$(q,\vec{N})$ and ASEP$_\mathsf{L}(q,\vec{N},t)$ (dynamic ASEP), the second for the duality between ASIP$(q,\vec{k})$ and ASIP$_\mathsf{L}(q,\vec{k},t)$ (dynamic ASIP). We refer to \cite{GroeneveltWagenaarDyn} and \cite{Wa} for detailed descriptions of the interacting particle systems and their relation with the quantum algebra $\U_q(\mathfrak{sl}_2)$. Having duality functions, other duality functions can be obtained from the scalar-product method \cite{CarFraGiaGroRed}. In this way it follows that ASEP$_\mathsf{L}(q,\vec{N},s)$ is dual to ASEP$_\mathsf{L}(q,\vec{N},t)$ with duality function
	\[
		\langle K_{1,s}(\,\cdot\,,\vec{x}), K_{v,t}(\,\cdot\,,\vec{y})  \rangle_{H_{\vec{N}}} = R_\mathrm{r}(\vec{x},\vec{y};s,t).
	\]
	Similarly, ASIP$_\mathsf{L}(q,\vec{k},s)$ is dual to ASIP$_\mathsf{L}(q,\vec{k},t)$ with duality function
	\[
	\langle \Phi_{1,s}(\,\cdot\,,\vec{x}), \Phi_{v,t}(\,\cdot\,,\vec{y})  \rangle_{H_{\vec{k}}} = P_\mathrm{r}(\vec{x},\vec{y};s,t).
	\]

\end{document}